\documentclass[11pt,twoside]{amsart}
\usepackage{mathrsfs}
\usepackage{txfonts}
\usepackage{bbm}
\usepackage{amsfonts}
\usepackage{amsmath}
\numberwithin{equation}{section}
\usepackage{esint}
\usepackage{amssymb}
\usepackage{wasysym}
\usepackage{appendix}
\usepackage{psfrag}
\usepackage{graphics,epsfig,amsmath}
\usepackage{comment}
\usepackage{enumerate}
\usepackage{times}
\usepackage{fancyhdr}
\usepackage{cite}

\newtheorem{Theorem}{Theorem}[section]
\newtheorem{Lemma}[Theorem]{Lemma}
\newtheorem{Proposition}{Proposition}[section]
\newtheorem{Definition}{Definition}[section]
\newtheorem{remark}[Theorem]{Remark}

\newtheorem{Question}{Question}[section]

\usepackage{metalogo}
\usepackage[cmyk]{xcolor}
\selectfont
\usepackage{mathtools}
\mathtoolsset{showonlyrefs}
\renewcommand{\epsilon}{\varepsilon}
\usepackage{geometry}
\title[Uniqueness and Nondegeneracy of Positive Solutions]
{Uniqueness and Nondegeneracy of ground states of $ -\Delta u +  (-\Delta)^s u+u = u^{p+1} \quad \hbox{in $\mathbb{R}^n$}$ when $s$ is close to $0$ and $1$}

\author[X. Su]{Xifeng Su}
\address{School of Mathematical Sciences, Laboratory of Mathematics and Complex Systems (Ministry of Education)\\
	Beijing Normal University,
	No. 19, XinJieKouWai St., HaiDian District, Beijing 100875, P. R. China}
\email{xfsu@bnu.edu.cn, billy3492@gmail.com}

\author[C.Zhang]{Chengxiang Zhang}
\address{School of Mathematical Sciences, Laboratory of Mathematics and Complex Systems (Ministry of Education)\\
	Beijing Normal University,
	No. 19, XinJieKouWai St., HaiDian District, Beijing 100875, P. R. China}
\email{zcx@bnu.edu.cn}

\author[J. Zhang]{Jiwen Zhang}
\address{School of Mathematical Sciences, Laboratory of Mathematics and Complex Systems (Ministry of Education)\\
	Beijing Normal University,
	No. 19, XinJieKouWai St., HaiDian District, Beijing 100875, P. R. China}
\email{jwzhang826@mail.bnu.edu.cn,jiwen.zhang@uwa.edu.au}


\begin{document}
	\maketitle
	
	\begin{abstract}
		 We are concerned with the  mixed local/nonlocal Schr\"{o}dinger equation
		\begin{equation}
			- \Delta u +  (-\Delta)^s u+u = u^{p+1} \quad \hbox{in $\mathbb{R}^n$,}
		\end{equation}
		for arbitrary space dimension $n\geqslant1$, $s\in(0,1)$, and $p\in(0,2^*-2)$ with $2^*$ the critical Sobolev exponent. 
		
		We provide the existence  and several fundamental properties of nonnegative solutions for the above equation. And then, we prove that, if $s$ is close to $0$ and $1$, respectively,  such equation then possesses a unique (up to translations) ground state, which is nondegenerate.				
	\end{abstract}

	\tableofcontents
	
	\section{Introduction}

%
	
	The main purpose of this paper is to derive some uniqueness and nondegeneracy results for positive solutions of the following Schr\"{o}dinger equation driven by  the mixed local/nonlocal operator:
	 \begin{equation}\label{main equation}
	 \begin{cases}
	 		- \Delta u +  (-\Delta)^s u+u = u^{p+1} \quad \hbox{in $\mathbb{R}^n$,}\\
	 		u\in H^1(\mathbb{R}^n),  \quad u\geqslant 0.
	 \end{cases}
	  \end{equation}
   Here, the space dimension $n\geqslant1$ could be arbitrary and $s$ could be any real number in $(0,1)$, and the fractional Laplacian $ (-\Delta )^{s} $ is defined as
  \begin{equation}
  	(-\Delta )^{s} u(x)=c_{n,s}\,\text{PV}\int_{\mathbb{R}^{n}} \frac{u(x)-u(y)}{|x-y|^{n+2s}}\, dy,
  \end{equation}
  where $ c_{n,s} $ is a dimensional constant that depends on $n$ and $s$, precisely given
  \[ c_{n,s}=\left(\int_{{\mathbb{R}^n}} \frac{1-\cos(\xi_1)}{|\xi|^{n+2s}}\, d\xi\right)^{-1}.\]
  In view of \cite[Proposition~4.1]{MR2944369}, we have 
  \begin{equation}\label{c_n,s}
  	\lim\limits_{s\rightarrow 1^{-}}\frac{c_{n,s}}{s (1-s)}=\begin{cases}
  		\,\frac{4n}{\omega_{n-1} }\quad &\text{if } n\geqslant 2\\
  		2 & \text{if } n=1, 
  	\end{cases}  \quad \text{and} \quad
  	\lim\limits_{s\rightarrow 0^{+}}\frac{c_{n,s}}{s (1-s)}=\begin{cases}
  		\,\frac{2}{\omega_{n-1} }	\quad &\text{if } n\geqslant 2\\
  		\,1 & \text{if } n=1.
  	\end{cases}
  \end{equation}
  

  Here and throughout the following, we assume that the exponent in the nonlinearity satisfies $0<p<2^*-2, $
 where \begin{equation}
  	2^*:=\begin{cases}
  		\frac{2n}{n-2} \quad & \text{for }n>2,\\
  		+\infty &\text{for } n=1,2,
  	\end{cases}
  \end{equation}
   is the critical Sobolev exponent of the embedding $H^1\hookrightarrow L^{p+2}$.

   In this framework, the classical local diffusion induced by Brownian motions is replaced by a mixture of local and non-local diffusions  driven by a Brownian motion and an independent symmetric $2s$-stable process, which is a general Markov process with both a continuous component and a discontinuous component. These types of diffusive operators arise in both pure mathematical research and real world applications and receive a great attention recently, see for instance \cite{llaveV09, BCCI12, CKS12, BD13, DV21, BVDV21, BDVV22b,  DSVZ24, SVWZ22, MR4803021, MR4808805} and so on.
   
   This superposition of stochastic processes in the quantum action functional gives
   rise precisely to the mixed local/nonlocal Schr\"{o}dinger equation in \eqref{main equation} while the background of  the non-relativistic quantum mechanics formulated by Feynman and Hibbs \cite{FHbook} as a path integral over the Brownian paths  leads to the standard Schr\"{o}dinger equation.
   
   
 The nondegeneracy and uniqueness results of the standard Schr\"{o}dinger equation have been established in \cite{Coffman72, MR886933, MR969899, MR1201323}, and these  conclusions were extended into the nonlocal case in \cite{AT91, MR3070568, MR3207007, MR3530361} and references therein.
In contrast to both situations above with single scale, 
it seems  that extremely little is known about these properties in a mixed local/nonlocal framework that is not scale-invariant.  
Here, the simultaneous presence of a leading local
operator, and a lower order fractional one, constitutes the essence of the matter.

Following the spirit of \cite{MR3207007,MR3070568,MR3530361} and exploring the essential properties of the present mixed operator, we take an approach that relies more
on the estimates available in regularity theory of local operators and succeed in obtaining some uniform  estimates for solutions  independent of $s$ and then establish the nondegeneracy and uniqueness results.

To well formulate the main results in this paper, let us first introduce the corresponding minimization problem of the  energy functional of 
\[ 
\mathcal{F}_s(u):=\frac{1}{2}\left[\int_{{\mathbb{R}^{n}}}\left(|\nabla u|^2+|(-\Delta)^{s/2}u|^2+|u|^2\right)\, dx\right]-\frac{1}{p+2}\int_{{\mathbb{R}^{n}}}|u|^{p+2}\, dx \qquad \forall u\in H^1(\mathbb{R}^n)
\]
over the Nehari manifold
\[ \mathcal{B}_s:=\left\{u\in H^1(\mathbb{R}^n):u\not\equiv 0 \text{ and }\, \left\langle \nabla \mathcal{F}_s(u),u \right\rangle_{H^1}=0 \right\}. \]

One may observe that, by some suitable rescaling, to find a minimizer of the above minimization problem of $\mathcal{F}_s$ over $\mathcal{B}_s $ could be transformed into finding
a minimizer $u_s\in H^1(\mathbb{R}^n)$ of a minimization problem of the following functional 
\begin{equation}\label{j}
	J_s(v):=\frac{\int_{{\mathbb{R}^{n}}}\left(|\nabla v|^2+|(-\Delta)^{s/2}v|^2+|v|^2\right)\, dx}{\left(\int_{{\mathbb{R}^{n}}}|v|^{p+2}\, dx\right)^{2/(p+2)}}  \qquad \text{ for any } v\in H^1(\mathbb{R}^n) \text{ with } v\not\equiv 0. 
\end{equation}
Note that  the functional $J_s$ has indeed minimizers  $u_s\in H^1(\mathbb{R}^n)$ due to  the concentration-compactness principle (see e.g. \cite{MR778974}).

In other words, the above minimizers $u_s$ attain the greatest lower bound:
\begin{equation}\label{lambda s}
	\lambda_s:=\inf_{v\in H^1(\mathbb{R}^n)\backslash\left\{0\right\}}J_s(v)=\inf_{
		v\in H^1(\mathbb{R}^n)
		\atop 
		\|v\|_{L^{p+2}(\mathbb{R}^n)}=1
		}\int_{{\mathbb{R}^{n}}}\left(|\nabla v|^2+|(-\Delta)^{s/2}v|^2+|v|^2\right)\, dx.
\end{equation} 

From \cite{DSVZ24}, the following existence result along with fundamental properties of nonnegative solutions for equation~\eqref{main equation} can be summarized below:  
\begin{Theorem}\label{th: existence and properity} 
Let $n\geqslant 1, s\in(0,1)$, and $p\in(0,2^*-2)$. Then the following holds:
	\begin{itemize}
		\item[(i)]\text{EXISTENCE}: There exists a nonnegative solution  $u_s\in H^1(\mathbb{R}^n)$ of equation \eqref{main equation}, which is a minimizer of the above functional  $J_s$.
		\item[(ii)] REGULARITY, SYMMETRY AND DECAY: If $u_s\in H^1(\mathbb{R}^n)$ with $u_s\geqslant 0$ and $u_s\not\equiv 0$ solves equation~\eqref{main equation}, then 
		\begin{itemize}
		\item there exists some $x_0\in\mathbb{R}^n$ such that $u_s(\cdot-x_0)$ is radial, positive, and decreasing in $|x-x_0|$. 
		\item Moreover, the function $u_s\in C^{2}(\mathbb{R}^n)\cap H^2(\mathbb{R}^n)$ and there exist two suitable constants $C_2\geqslant C_1>0$ such that 
		\[ \frac{C_1}{|x|^{n+2s}}\leqslant u_s(x)\leqslant  \frac{C_2}{|x|^{n+2s}} \quad \forall |x|>1. \]
		\end{itemize}
	\end{itemize}
\end{Theorem}
For the reader's  convenience, we  present the  proof of Theorem~\ref{th: existence and properity} in Appendix~\ref{sec:Existence and Properties of Ground states }.


\smallskip

It would then be interesting to introduce  the following notion of ground states for  those positive solutions of \eqref{main equation} obtained in  Theorem~\ref{th: existence and properity}.
\begin{Definition}[Ground states]\label{ground states}
	We say that $u_s$ is a ground state of equation~\eqref{main equation} if $u_s\geqslant 0$ with $u_s\not\equiv 0$ and $u_s$ has the least energy, namely, ${J}_s(u_s)=\lambda_s $.
\end{Definition}

We remark that $u_s$ is a ground state of equation~\eqref{main equation} if and only if $u_s$ is a minimizer of functional $\mathcal{F}_s$ on $\mathcal{B}_s$. 
Let us denote the set of ground states of equation~\eqref{main equation} for $s\in(0,1]$ by 
\begin{equation}\label{m_s}
	\mathcal{M}_s:=
	\left\{\begin{array}{ll}
	u_s\in H^1(\mathbb{R}^n):   u_s \text{ solves equation } \eqref{main equation} \text{ with } u_s\geqslant 0, \text{ and } J_s(u_s)=\lambda_s
	\end{array}
	\right\}.
\end{equation}

Then, for any $s\in(0,1)$, we have  the following fundamental facts about $\mathcal{M}_s$.
\begin{remark}\label{reamrk 1.2}
\begin{itemize}
\item Let $u_s\in \mathcal{M}_s$, then $u_s^p\in K_s$, i.e., the potential $V =u_s^p$ belongs to the ``Kato-class" (Definition \ref{definition of Kato}) with respect to
$ - \Delta  +  (-\Delta)^s $, which follows from Lemma~\ref{lemma kato lp}. 
\item	 Let  $u_s\in \mathcal{M}_s$, then, from Theorem~\ref{th: existence and properity},  one can assume that $u_s=u_s(|x|)>0$ is radial and decreasing in $|x| $.
\end{itemize}
\end{remark}

For any $u_s\in\mathcal{M}_s$, we  denote  the corresponding linearized operator of $J_s$ at $u_s$ by
\begin{equation}\label{Ls}
	L_s= - \Delta  +  (-\Delta)^s +1 - (p+1)u_s^{p}
\end{equation}
acting on $L^2(\mathbb{R}^n)$. Note that
 $L_s$ is a self-adjoint operator on $L^2(\mathbb{R}^n)$ with quadratic-form domain $H^1(\mathbb{R}^n)$ and
 operator domain $H^2(\mathbb{R}^n)$.

Therefore, we obtain the following nondegenerate property of elements in $\mathcal{M}_s$.
\begin{Theorem}\label{th:nondegeneracy}
	Let $n\geqslant 1 $. There exist $s_0 , s_1\in(0,1)$ such that for every  $s\in(0,s_0)\cup (s_1,1)$,  if $u_s\in \mathcal{M}_s$, then
the linearized operator $L_s$ is nondegenerate; i.e., its kernel is exhausted by the derivatives of $u_s$ and their linear combinations:
	\[  \text{Ker}\,L_s=\text{span}\left\{\partial_{i} u_s, i=1,\cdots,n\right\}. \]
\end{Theorem}

 We remark that the above nondegeneracy result for space dimension $n=1$  is established by 
  following the strategy developed in~\cite{MR3070568}, i.e., 
 combining heat kernel with Perron-Frobenius arguments, while  
  the approach of proving the case $n\geqslant 2$ is based on  a local realization   and  the decomposition in terms of spherical harmonics.

Our next result is a uniqueness property.
\begin{Theorem}\label{th:uniqueness}
	Let $n\geqslant 1 $. There exist $s_0 , s_1\in(0,1)$ such that for every  $s\in(0,s_0)\cup (s_1,1)$,  the ground state $u_s\in H^1(\mathbb{R}^n)$ of equation~\eqref{main equation} is unique up to a translation.
\end{Theorem}

\begin{remark} Let $n\geqslant 1 $. Then, there exist $s_0 , s_1\in(0,1)$ such that for every  $s\in(0,s_0)\cup (s_1,1)$,
the minimizer of $\mathcal{F}_s$ on $\mathcal{B}_s$ is unique up to a translation, while the  minimizer of $J_s$ is unique up to a rescaling: $u_s(x)\mapsto au_s(x-b)$ for some $a>0, b\in\mathbb{R}^n$.
\end{remark}


On the other hand, we observe that, according to Proposition \ref{proposition A.1}, the ground state \( u_s \) has Morse index\footnote{The associated linearized operator $L_s$ has exactly one strictly negative eigenvalue (counting multiplicity). } equal to $1$. These observations lead to the following questions:

	\begin{Question}\label{question1}
	Is the uniqueness result of minimizers  available for any $s\in(0,1)$?
\end{Question}

	\begin{Question}\label{question}
	Is the uniqueness result available for solutions with Morse index equal to $1$ as in \cite{MR3530361} in the mixed setting?
\end{Question}


	
	
%


 This paper is organized as follows. In Section~\ref{sec:Preliminaries}, we introduce the working space and  several  uniform estimates and asymptotics of the solutions for equation~\eqref{main equation}.
 
 Next, Section~\ref{sec:Nondegeneracy} is devoted to establish the nondegeneracy result of ground states as stated in Theorem~\ref{th:nondegeneracy}. Finally, a uniqueness property of ground states described in Theorem~\ref{th:uniqueness} is shown in Section~\ref{sec:uniqueness}. 

In the appendix, we collect  several technical results and proofs needed in this paper.

\section{Preliminaries}\label{sec:Preliminaries}

As regards the functional framework, suppose that $s\in[0,1]$ and $n\geqslant 1$, we consider  here the classical Sobolev space
\begin{equation}
	H^1(\mathbb{R}^n):=\left\{u\in L^2(\mathbb{R}^n): \int_{\mathbb{R}^n}\left(1+|\xi|^2\right)|\hat{u}|^2\, d\xi<\infty \right\}
\end{equation}
and equip it  with an  equivalent norm 
\[ \|u\|_{s}^2:=\int_{\mathbb{R}^n}\left(1+|\xi|^2+|\xi|^{2s}\right)|\hat{u}|^2\, d\xi,\]
where, as usual, $\hat{u}$ is the Fourier transform of function $u$, namely,
$$\hat{u}(\xi):=\left(2\pi\right)^{-\frac{n}{2}}\int_{{\mathbb{R}^n}}e^{-i\xi\cdot x}u(x)\, dx.$$

Since the operator $-\Delta+(-\Delta)^s $  may be defined using Fourier transform as
\[ \widehat{(-\Delta)\varphi} (\xi)+\widehat{(-\Delta)^s\varphi} (\xi):=|\xi|^2\hat{\varphi}(\xi)+|\xi|^{2s}\hat{\varphi}(\xi) \quad \text{for } \xi\in\mathbb{R}^n, \] 
for every $\varphi\in C_c^\infty(\mathbb{R}^n)$, we are in a position to apply Plancherel's formula and adopt a weak notion of solution $u\in H^1(\mathbb{R}^n)$ for equation~\eqref{main equation} by the identity
\begin{equation}\label{weak solution}
	\int_{{\mathbb{R}^n}}\left(1+|\xi|^2+|\xi|^{2s}\right)\hat{u}\bar{\hat{\varphi}}\, d\xi=\int_{{\mathbb{R}^n}}u^{p+1}\varphi\, dx\quad \text{for any } \varphi\in H^1(\mathbb{R}^n).
\end{equation}

In addition, we define  the scalar product by
\begin{equation}\label{eq: inner product}
	\left\langle u,v \right\rangle_{s}:=\int_{\mathbb{R}^n}\left(1+|\xi|^2+|\xi|^{2s}\right)\hat{u}(\xi)\hat{v}(\xi)\, d\xi.
\end{equation}	
 In particular, we notice that,   in both cases  $s=0$ and $1$,  
 the corresponding infima \begin{equation}\label{lambda1}
 	\begin{split}
 	\lambda_0&=\inf_{u\in H^1(\mathbb{R}^n)\backslash\left\{0\right\}}J_0(u)= 2^{1-\frac{n}{2}+\frac{n}{p+2}}\inf_{u\in H^1(\mathbb{R}^n)\backslash\left\{0\right\}}\frac{\int_{{\mathbb{R}^n}}\left(1+|\xi|^2\right)|\hat{u}|^2\, d\xi}{\|u\|^2_{L^{p+2}(\mathbb{R}^n)}},\\
 \text{and }	\qquad 	\lambda_1&=\inf_{u\in H^1(\mathbb{R}^n)\backslash\left\{0\right\}}J_1(u)= 2^{\frac{n}{2}-\frac{n}{p+2}}\inf_{u\in H^1(\mathbb{R}^n)\backslash\left\{0\right\}}\frac{\int_{{\mathbb{R}^n}}\left(1+|\xi|^2\right)|\hat{u}|^2\, d\xi}{\|u\|^2_{L^{p+2}(\mathbb{R}^n)}}.
 	\end{split}
\end{equation}
are bounded, where  the functionals $J_1$  and $J_0$  defined in~\eqref{j}.
 Also, for $i\in\left\{0,1\right\}$,  there exists a positive, radial and decreasing in $r=|x|$ minimizer $u_i$  in $H^1(\mathbb{R}^n)$ with $\|u_i\|_{L^{p+2}(\mathbb{R}^n)}=1$, namely,
$\lambda_i=J_i(u_i).$

%


 We now  provide several  uniform estimates and asymptotics  of the solutions of \eqref{main equation}  in $\mathcal{M}_s$.

\begin{Lemma}\label{lemma lambda1}
	Let $\lambda_1$ and $\lambda_0$ be defined in \eqref{lambda1}, we  then have 
 \begin{equation}\label{vdsfsvd}
 	\inf\limits_{s\in[0,1]}\inf\limits_{u_s\in \mathcal{M}_s}\|u_s\|^2_{s}> C_1(p,\lambda_1,\lambda_0) \quad \text{and} \quad \sup\limits_{s\in[0,1]}\sup\limits_{u_s\in \mathcal{M}_s}\|u_s\|^2_{s}<C_2(p,\lambda_1,\lambda_0).
 \end{equation}
\end{Lemma}
\begin{proof} 
Assume that $s\in(0,1)$ and $u_s\in \mathcal{M}_s$ . From the definition of $\mathcal{M}_s$  and~\eqref{weak solution}, it follows that 
$\|u_s\|_{s}^2=\lambda_s^{1+\frac{2}{p}}$. Let $u_1\in \mathcal{M}_1 $, recalling~\eqref{lambda1} and the definition of $\lambda_s$,  one has that
	\begin{equation}
2^{-\frac{n}{2}+\frac{n}{p+2}}\lambda_1	\leqslant\lambda_s\leqslant 	\frac{\|u_1\|_s^2}{\|u_1\|^2_{L^{p+2}(\mathbb{R}^n)}}\leqslant 2 \frac{\int_{{\mathbb{R}^n}}\left(1+2|\xi|^2\right)|\widehat{u_1}|^2\, d\xi}{\|u_1\|^2_{L^{p+2}(\mathbb{R}^n)}}=2J_1(u_1)=2\lambda_1.
	\end{equation}
This implies the desired result~\eqref{vdsfsvd}.
\end{proof}

\begin{Lemma}\label{lemma uniform bound}
We have that
\begin{equation}\label{fbds}
	0<\Lambda_0:=\sup\limits_{s\in[0,1]} \sup\limits_{u_s\in \mathcal{M}_s}\|u_s\|_{L^\infty(\mathbb{R}^n)}< C_0,
\end{equation}
where $C_0$ only depends  on $n,p,\lambda_1$ and $\lambda_0$.
\end{Lemma}
\begin{proof}
Note that to prove \eqref{fbds}, it suffices to show that $\Lambda_{0}$  is finite since
	\begin{equation}
	\Lambda_{0}\geqslant\sup\limits_{u_1\in\mathcal{M}_1}\|u_1\|_{L^\infty(\mathbb{R}^n)}>0.
	\end{equation} 
For this, we make use of  a standard Moser iteration argument. Indeed, assume that  $u_s\in \mathcal{M}_s$ satisfies 
\begin{equation}\label{u_s}
	-\Delta u_s+(-\Delta)^s u_s+u_s=u_s^{p+1}.
\end{equation}
For some $k>0$, let us denote
\begin{equation}
Q_{s,k}=\begin{cases}
	u_s \quad &\text{ if } u_s<k\\
	k &\text{ if } u_s\geqslant k.
\end{cases}
\end{equation}
We then observe that $D Q_{s,k}=0$ in $\left\{u_s\geqslant k\right\}$.
Set the text function 
\begin{equation}\label{sgs}
	\varphi:= Q_{s,k}^\alpha u_s\in H^1(\mathbb{R}^n)
\end{equation}
for suitable constant  $\alpha\geqslant 0$. 
Direct calculation yields that for every $ x, y\in\mathbb{R}^n$
\begin{equation}\label{sda}
	\varphi(x)\leqslant \varphi(y) \quad \text{if } u_s(x)\leqslant u_s(y), 
\end{equation}
\begin{equation}\label{varphi}
\text{ and }\quad 	D\varphi(x)= Q_{s,k}^\alpha(x) D u_s(x)+\alpha Q_{s,k}^{\alpha}(x) D Q_{s,k}(x).
\end{equation}

 We multiply the equation of~\eqref{u_s} by $\varphi$ and integrate over $\mathbb{R}^n$, one has that
 \begin{equation}\label{sfssds}
 	\int_{\mathbb{R}^n} \left(D {u_s} \cdot D\varphi+ u_s\cdot\varphi\right)\, dx +\int_{\mathbb{R}^{2n}}\frac{(u_s(x)-u_s(y))(\varphi(x)-\varphi(y))}{|x-y|^{n+2s}}\, dxdy
	=\int_{\mathbb{R}^n} {u_s}^{p+1} \varphi\, dx.
 \end{equation}
From~\eqref{sda}, we know that the second term on the left side of \eqref{sfssds} above is nonnegative. Thus, by combining~\eqref{varphi} with~\eqref{sfssds}, one deduces that
\begin{equation}
	\begin{split}\label{dsvs}
		&\int_{\mathbb{R}^n} Q_{s,k}^\alpha(x)|D {u_s}(x)|^2 \, dx+\int_{\mathbb{R}^n}  \alpha Q_{s,k}^{\alpha}(x) |D Q_{s,k}(x)|^2\, dx+\int_{\mathbb{R}^n} Q_{s,k}^\alpha(x) u_s^2(x)\,  dx\\
		\leqslant&\int_{\mathbb{R}^n} {u_s}^{p} Q_{s,k}^\alpha(x) u_s^2(x)\, dx.
	\end{split}
\end{equation}
Moreover, we denote $w_s:=Q_{s,k}^{\frac{\alpha}{2}}u_s$. Observing that
\[ |Dw_s|^2= Q_{s,k}^{{\alpha}} |Du_s|^2+\left(\alpha+\frac{\alpha^2}{4}\right) Q_{s,k}^{{\alpha}}|DQ_{s,k}|^2 \leqslant \left(1+\alpha\right)\left(Q_{s,k}^{{\alpha}} |Du_s|^2+\alpha Q_{s,k}^{{\alpha}}|DQ_{s,k}|^2\right).\]
As a consequence of this and \eqref{dsvs}, one has that 
\begin{equation}\label{sfdsfd}
		\int_{\mathbb{R}^n} |Dw_s|^2\, dx\leqslant (1+\alpha)\left(\int_{\mathbb{R}^n}Q_{s,k}^{{\alpha}} |Du_s|^2\, dx+\int_{\mathbb{R}^n}\alpha Q_{s,k}^{{\alpha}}|DQ_{s,k}|^2\, dx\right)\leqslant \left(1+\alpha\right)\int_{\mathbb{R}^n} {u_s}^{p} w_s^2\, dx.
\end{equation}
Owing to the H\"{o}lder inequality, we see that 
\begin{equation}\label{fvdfa}
	\int_{\mathbb{R}^n} {u_s}^{p} w_s^2\, dx\leqslant \|u_s\|^p_{L^{q_0}(\mathbb{R}^n)}\|w_s\|^2_{L^{\frac{2q_0}{q_0-p}}(\mathbb{R}^n)},
\end{equation}
 where $q_0$ is either equal to the critical exponent $2^*$ if $n > 2$ or any real number in $(p+2,+\infty)$
	if $n =1,2$.  Since $0<p<2^*-2$, one then has that  $q_0>\frac{2q_0}{q_0-p}>2$ for every $n\geqslant 1$.  By the Gagliardo-Nirenberg interpolation inequality and Young's inequality, one has that
\begin{equation}
	\begin{split}\label{bkhvh}
		\|w_s\|_{L^{\frac{2q_0}{q_0-p}}(\mathbb{R}^n)}
		&\leqslant C\|Dw_s\|^\frac{pn}{2q_0}_{L^{2}(\mathbb{R}^n)}\|w_s\|^{1-\frac{pn}{2q_0}}_{L^{2}(\mathbb{R}^n)}\\
		&\leqslant C\left(\epsilon\|Dw_s\|_{L^{2}(\mathbb{R}^n)} +\epsilon^{-\frac{pn}{2q_0-pn}}\|w_s\|_{L^{2}(\mathbb{R}^n)}\right)
	\end{split}
\end{equation}
where the positive constant $C$ only depends on  $n,q_0,p$.

Taking $\epsilon>0$ small enough, by combinning ~\eqref{sfdsfd}, \eqref{fvdfa} and
\eqref{bkhvh}, one obtains that
\begin{equation}\label{sdds}
	\int_{\mathbb{R}^n} |Dw_s|^2\, dx\leqslant C(1+\alpha)^{\frac{2q_0}{2q_0-pn}}\int_{\mathbb{R}^n}|w_s|^2\, dx
\end{equation}
where the positive constant $C$ only depends on ${n,q_0,\|u_s\|_{s}}$.

Using the Gagliardo-Nirenberg interpolation inequality again, from~\eqref{sdds}, it follows that 
\begin{equation}\label{chjv}
	\left(\int_{\mathbb{R}^n} |w_s|^{q_0}\, dx\right)^{\frac{2}{q_0}}\leqslant C(1+\alpha)^{\frac{2q_0}{2q_0-pn}}\int_{\mathbb{R}^n}|w_s|^2\, dx
\end{equation} 
for renaming $C$ which only depends on ${n,q_0,\|u_s\|_{s}}$. Furthermore, denote $\beta:=\frac{2q_0}{2q_0-pn}>1 $ and $\gamma:=\frac{q_0}{2}>1$, thus the inequality~\eqref{chjv} can be rewritten as
\begin{equation}\label{dsa}
	\left(\int_{\mathbb{R}^n} |w_s|^{2\gamma}\, dx\right)^{1/\gamma}\leqslant C(1+\alpha)^{\beta}\int_{\mathbb{R}^n}|w_s|^2\, dx.
\end{equation} 
From the definition of $w_s$, it is immediate to check that 
\begin{equation}\label{sdfsg}
	\left(\int_{\mathbb{R}^n} |Q_{s,k}|^{(\alpha+2)\gamma}\, dx\right)^{1/\gamma}\leqslant C(1+\alpha)^{\beta}\int_{\mathbb{R}^n}|u_s|^{\alpha+2}\, dx.
\end{equation} 
Moreover, set $\eta:=\alpha+2\geqslant 2$, we have that 
\begin{equation}\label{dvdsf}
	\left(\int_{\mathbb{R}^n} |Q_{s,k}|^{\eta\gamma}\, dx\right)^{1/\gamma}\leqslant C(\eta-1)^{\beta}\int_{\mathbb{R}^n}|u_s|^{\eta}\, dx,
\end{equation} 
provided the integral in the right hand side is bounded. Here the constant $C$ is independent of $\eta$ and $k$.  We are on a position of taking $k\rightarrow \infty$, obtaining that 
\begin{equation}\label{sdadas}
 \|u_s\|_{L^{\eta\gamma}(\mathbb{R}^n)}\leqslant C^{1/\eta}\eta^{\beta/\eta}\|u_s\|_{L^\eta(\mathbb{R}^n)},
\end{equation} 
provided $\|u_s\|_{L^{\eta}(\mathbb{R}^n)}$ is bounded. The above estimate allows us to iterate, beginning with $\eta=2$. Now set for $i=0,1,2,3,\cdots$, $\eta_0=2$ and $\eta_{i+1}= \eta_i\gamma$.

As a result, by iteration we have that
\begin{equation}\label{vjh}
	\|u_s\|_{L^{\eta_{i+1}}(\mathbb{R}^n)}\leqslant C^{\sum\frac{i}{\gamma^i}}\|u_s\|_{L^2(\mathbb{R}^n)},
\end{equation} 
where $C$  depends on $n,p,\|u_s\|_{s},\gamma,\beta$. Since $\gamma>1$, one has that $\sum_{i=0}^{\infty}\frac{i}{\gamma^i}$ is bound. Let $i\rightarrow \infty$, we conclude that 
\[ \|u_s\|_{L^\infty(\mathbb{R}^n)}\leqslant C\|u_s\|_{L^2(\mathbb{R}^n)}, \]
where $C$ only depends on $n,p,\|u_s\|_{s}.$ 
As a consequence of this and Lemma~\ref{lemma lambda1}, we obtain \eqref{fbds}, as desired.
\end{proof}

Let us denote  by  $\|u\|^2_{2}:=\int_{{\mathbb{R}^n}}\left(1+|\xi|^2+|\xi|^{4}\right)|\hat{u}|^2\, d\xi$ the $H^2$-norm  of every $u\in H^2(\mathbb{R}^n)$.

\begin{Lemma}\label{lemma 2 norm is bound}
 We have that 
\begin{equation}\label{H2 norm}
	\sup\limits_{s\in[0,1]} \sup\limits_{u_s\in \mathcal{M}_s}\|u_s\|_{2}< C_3,
\end{equation}
for some positive constant $C_3$ only depending on $p,\lambda_1,\lambda_0, n$.
\end{Lemma}
\begin{proof}
	Let $u_s\in \mathcal{M}_s$ and $h_s:=u_s^{p+1}-u_s.$ By combining Lemmata~\ref{lemma lambda1} with~\ref{lemma uniform bound}, one has that 
	\begin{equation}\label{dfsvfv}
		\begin{split}
			\|h_s\|_{L^2(\mathbb{R}^n)}&\leqslant \|u_s\|_{L^2(\mathbb{R}^n)}+\|u_s^{p+1}\|_{L^2(\mathbb{R}^n)}\\
			&\leqslant \left(1+\|u_s\|^p_{L^\infty(\mathbb{R}^n)}\right)\|u_s\|_{L^2(\mathbb{R}^n)}\leqslant C(p,\lambda_1,n,\lambda_0).
		\end{split}
	\end{equation} 
 Moreover, since $u_s$ solves equation~\eqref{main equation}, we have that 
\begin{equation}\label{dfs}
	-\Delta u_s+(-\Delta)^s u_s=h_s.
\end{equation}
 In view of Theorem~\ref{th: existence and properity}, we know that $u_s\in C^2(\mathbb{R}^n)$. Thus, \eqref{dfs} holds pointwise. Owing to  \eqref{dfsvfv} and using Lemma~\ref{lemma lambda1} again,  one deduces that
 \begin{equation}
 \begin{split}
 	\|u_s\|_{2}^2&=\int_{{\mathbb{R}^n}}\left(1+|\xi|^2+|\xi|^4\right)|\widehat{u_s}|^2\, d\xi\\
 &\leqslant \|u_s\|^2_{s}+\int_{\mathbb{R}^{n}}|\xi|^4 |\widehat{u_s}|^2\, d\xi\leqslant \|u_s\|^2_{s}+\|h_s\|^2_{L^2(\mathbb{R}^n)}\leqslant C
 \end{split}
 \end{equation}
 for some positive constant $C$ only depending on  $p,\lambda_1,n,\lambda_0$, and the desired result~\eqref{H2 norm} plainly follows.
\end{proof}

%
%

\begin{Lemma}\label{Lemma gamma=0}
Let $d\in(0,1)$ and $s\in (0,d]$, then for every $u_s\in \mathcal{M}_s$
	\begin{equation}\label{bjkdf}
	\|(-\Delta)^{s} u_s- u_s\|_{L^2(\mathbb{R}^n)}\leqslant C|s|
\end{equation}
for a suitable constant $C>0$ independent of $s$.
\end{Lemma}
\begin{proof} We observe that 
	\begin{equation}\label{bjk}
		\begin{split}
			\int_{\mathbb{R}^{n}}\left|(-\Delta)^{s} u_s(x)- u_s\right|^2\, dx&= \int_{|\xi|\geqslant 1}\left||\xi|^{2s}-1\right|^2|\hat{u}_s|^2\, d\xi+ \int_{|\xi|< 1}\left||\xi|^{2s}-1\right|^2|\hat{u}_s|^2\, d\xi\\
			& =\int_{|\xi|\geqslant 1}\left|e^{2s\ln|\xi|}-1\right|^2|\hat{u}_s|^2\, d\xi+ \int_{|\xi|< 1}\left|1-e^{-2s\ln\frac{1}{|\xi|}}\right|^2|\hat{u}_s|^2\, d\xi\\
			&\leqslant 4s^2\int_{|\xi|\geqslant 1}(|\xi|^{2s}\ln|\xi|)^2\,|\hat{u}_s|^2\, d\xi+4s^2 \int_{|\xi|< 1}\left(\ln\frac{1}{|\xi|}\right)^2|\hat{u}_s|^2\, d\xi\\
			&\leqslant 4s^2(1-d)^{-2}\,\|u_s\|^2_{2}+4s^2\int_{|\xi|< 1}{16e^{-2}|\xi|^{-1/2}}|\hat{u}_s|^2\, d\xi\\
			&\leqslant 4s^2\left((1-d)^{-2}\|u_s\|^2_{2}+16e^{-2}c_n\|\hat{u}_s\|^2_{L^\infty(\mathbb{R}^n)}\right)\\
			&\leqslant 4s^2\left((1-d)^{-2}\|u_s\|^2_{2}+16e^{-2}c_n\|{u}_s\|^2_{L^1(\mathbb{R}^n)}\right).
		\end{split}
	\end{equation}
We now  claim that \begin{equation}\label{fbe}
	\sup\limits_{s\in[0,1]}\|u_s\|_{L^1(\mathbb{R}^n)}\leqslant \|u_s\|_{s}^2.
\end{equation}
Indeed, for every $s\in(0,1)$, recall that $u_s\in L^1(\mathbb{R}^n)\cap C^2(\mathbb{R}^n)$ holds by Theorem~\ref{th: existence and properity}.  Integrating equation~\eqref{main equation}, and taking into account the fact that $\int_{{\mathbb{R}^{n}}}(-\Delta)^su_s=0$, we  obtain that
\begin{equation}\label{vdsvf}
	\int_{{\mathbb{R}^{n}}}-\Delta u_s+\int_{{\mathbb{R}^{n}}}u_s=\int_{{\mathbb{R}^{n}}}u_s^{p+1},
\end{equation}
which implies that  $-\Delta u_s\in L^1(\mathbb{R}^n) $. Moreover, we notice that, when $n\geqslant 2$, 
\begin{equation}\label{dfws}
	\begin{split}
		\int_{{\mathbb{R}^{n}}}-\Delta u_s&=\int_{S^{n-1}}\int_{0}^{\infty}r^{n-1} \left(-\partial_{rr}u_s(r)-\frac{n-1}{r}\partial_ru_s(r)\right)\, dr\, d\sigma\\
		&=-\omega_{n-1}\int_{0}^{\infty} \partial_r\left(r^{n-1}\partial_{r}u_s\right)\, dr=0,
	\end{split}
\end{equation}
thanks to the decay of $\nabla u_s$ at infinity (see for instance Lemma~\ref{lemma uniform decay of spectrum}). Similarly, it is immediate to check that, $\int_{{\mathbb{R}}}-\Delta u_s=0$ when $n=1$, thanks to the fact that $\partial_ru_s$ is odd. As a consequence of this, using the  H\"{o}lder inequality and~\eqref{vdsvf},  obtaining that 
\begin{equation}\label{vdsvfd}
\int_{{\mathbb{R}^{n}}}u_s=\int_{{\mathbb{R}^{n}}}u_s^{p+1}\leqslant \left(\int_{{\mathbb{R}^{n}}}u_s^{p+2}\right)^{\frac{p}{p+1}}\, \left(\int_{{\mathbb{R}^{n}}}u_s\right)^{\frac{1}{p+1}}=\|u_s\|_{s}^{\frac{2p}{p+1}}\, \left(\int_{{\mathbb{R}^{n}}}u_s\right)^{\frac{1}{p+1}}.
\end{equation}
This and Lemma~\ref{lemma lambda1} yield the claim~\eqref{fbe}. Thus,   by combining~\eqref{bjk} with~\eqref{fbe}, in virtue of Lemma~\ref{lemma 2 norm is bound}, we conclude the desired result~\eqref{bjkdf}.
\end{proof}

\begin{Lemma}\label{lemma convergence}
Fix $\gamma\in[0,1]$. Let $s_k\in(0,1)$ be such that $s_k\rightarrow \gamma$. Let $u_{s_k}\in\mathcal{M}_{s_k}$.
	Then, there exist $u_{\gamma}\in\mathcal{M}_{\gamma}$ and a  subsequence (still denoted by $s_k$) such that
	\begin{equation}
		\|u_{s_k}-u_\gamma\|_{2}\rightarrow 0 \quad \text{as } k\rightarrow \infty.
	\end{equation}
\end{Lemma}
\begin{proof}
From Remark~\ref{reamrk 1.2},  it follows that $u_{s_k}=u_{s_k}(|x|)>0$ is radial and decreasing in $|x| $. From this, we see that for any $R>0$,
	\begin{equation}
		\|u_{s_k}\|_{L^1(\mathbb{R}^n)}\geqslant \int_{|x|\leq R}|u_{s_k}|\, dx\geqslant C_{n} R^{n}u_{s_k}(R).
			\end{equation}
	Owing to~\eqref{fbe}, one has that 
	\begin{equation}\label{dgsdbf}
		u_{s_k}(|x|)\leqslant C|x|^{-n}
	\end{equation}
for some constant $C>0$ only depending on $n,p,\lambda_1,\lambda_0$.

We split this proof into the following steps.

Step 1. We prove strong convergence of  $\left\{u_{s_k}\right\}$ in $L^{p+2}(\mathbb{R}^n)$.

Indeed, since $u_{s_k}$ is bounded in $H^1(\mathbb{R}^n)$, there exists $u_\gamma$ such that $u_{s_k}\rightharpoonup u_\gamma$ weakly in $H^1(\mathbb{R}^n)$. Furthermore, by local Rellich compactness, we have that $u_{s_k}\rightarrow u_\gamma$ strongly in $L^{p+2}_{\rm loc}(\mathbb{R}^n)$ and pointwise a.e. in $\mathbb{R}^n$. Let $\epsilon>0$ be given, the  uniform decay estimate~\eqref{dgsdbf} and  Fatou Lemma imply that there exists $R_\epsilon>0$ sufficiently large that 
\begin{equation}\label{vdsv}
	\int_{|x|>R_\epsilon} |u_{s_k}|^{p+2}\, dx\leqslant \frac{\epsilon^{p+2}}{4} \quad \text{and }\quad \int_{|x|>R_\epsilon} |u_{\gamma}|^{p+2}\, dx\leqslant\frac{\epsilon^{p+2}}{4}.
\end{equation}
Moreover, since $u_{s_k}\rightarrow u_\gamma$ in $L^{p+2}_{\rm loc}(\mathbb{R}^n)$, there exists $k_0>0$ such that 
\begin{equation}\label{vsdsd}
	\int_{|x|\leqslant R_\epsilon} |u_{s_k}-u_{\gamma}|^{p+2}\, dx\leqslant\frac{\epsilon^{p+2}}{2}\quad \text{for every } k>k_0.
\end{equation}
By combining~\eqref{vdsv} and \eqref{vsdsd}, we thus conclude that
\begin{equation}\label{sdbsd}
	\|u_{s_k}-u_\gamma\|_{L^{p+2}(\mathbb{R}^n)}\leqslant \epsilon \quad \text{for all } k> k_0.
\end{equation}
 This implies that
 \begin{equation}\label{sdgdf}
 	u_{s_k}\rightarrow u_\gamma\quad  \text{ strongly in } L^{p+2}(\mathbb{R}^n).
 \end{equation} 

Step 2. We show that 
\begin{equation}\label{cdsvfd}
	u_{s_k}\rightarrow u_\gamma  \quad \text{in } H^1(\mathbb{R}^n).
\end{equation}
For this,  we observe that
 \begin{equation}\label{gdr}
 	\begin{split}
 		&\int_{\mathbb{R}^{n}} \nabla (u_{s_k}-u_\gamma)\nabla u_{s_k} +\int_{\mathbb{R}^{n}}(-\Delta)^{s_k} u_{s_k} (u_{s_k}-u_\gamma) +\int_{\mathbb{R}^{n}} u_{s_k} (u_{s_k}-u_\gamma)= \int_{\mathbb{R}^{n}}   u_{s_k}^{p+1}(u_{s_k}-u_\gamma).
 \end{split}
 \end{equation} 

Employing \cite[Lemma~2.4]{MR3207007} for $\gamma>0$ and Lemma~\ref{Lemma gamma=0} for $\gamma=0$, in virtue of lemma~\ref{lemma 2 norm is bound}, we see that 
\begin{equation}\label{edgsd}
\|(-\Delta)^{s_k} u_{s_k}-(-\Delta)^{\gamma} u_{s_k}\|_{L^2(\mathbb{R}^n)}\leqslant C_{\gamma}|\gamma-s_k|
\end{equation}
for a suitable $C_\gamma>0.$ Recalling Lemma~\ref{lemma lambda1}, 
by combining~\eqref{sdgdf}, ~\eqref{gdr} and~\eqref{edgsd}, it follows that 
\begin{equation}\label{vdsfd}
	\begin{split}
		\|u_{s_k}\|_{\gamma}^2&=
		\int_{\mathbb{R}^{n}} \nabla u_\gamma\nabla u_{s_k} +\int_{\mathbb{R}^{n}}(-\Delta)^{\gamma} u_{s_k} u_\gamma +\int_{\mathbb{R}^{n}} u_{s_k} u_\gamma+o(1)\\
		&=	\left\langle u_{s_k},u_\gamma \right\rangle_{\gamma}+o(1).
	\end{split}
\end{equation} 
Since $u_{s_k}$ weakly converge to $u_\gamma$ in $H^1(\mathbb{R}^n)$,  from~\eqref{vdsfd}, one deduces  that 
\begin{equation}
\lim\limits_{k\rightarrow\infty}	\|u_{s_k}\|_{\gamma}^2=\|u_{\gamma}\|_{\gamma}^2.
\end{equation}
This implies the desired result~\eqref{cdsvfd}.
 \smallskip

Step 3. We prove that $u_\gamma\in \mathcal{M}_\gamma$.

We notice that  for any $\varphi\in C_c^\infty(\mathbb{R}^n)$
 \begin{equation}\label{bsfgsd}
	 	\int_{\mathbb{R}^{n}} -\Delta \varphi\,  u_{s_k} +\int_{\mathbb{R}^{n}}(-\Delta)^{s_k} \varphi\,  u_{s_k} +\int_{\mathbb{R}^{n}} u_{s_k} \varphi = \int_{\mathbb{R}^{n}}   u_{s_k}^{p+1}\, \varphi.
	 \end{equation} 
By combining~\eqref{cdsvfd} with~\eqref{edgsd}, passing the limit in~\eqref{bsfgsd},  we obtain that $u_\gamma$ solves the following equation in the weak sense
\begin{equation}
	-\Delta u_\gamma + (-\Delta)^\gamma u_\gamma +u_\gamma=u_\gamma^{p+1}.
\end{equation}
By testing the equation against $u_\gamma$ itself, one has that $\|u_\gamma\|^2_{\gamma}=\|u_\gamma\|^{p+2}_{L^{p+2}}$.

To prove that $u_\gamma\in \mathcal{M}_\gamma$, it suffices to show that $J_{\gamma}(u_\gamma)=\lambda_\gamma.$ 
For this, we first consider the case $\gamma>0$.
 Let $s$, $\bar{s}\in (d,1]$ for some $d>0$, that will be taken one close to the other.
We might as well assume that $ \lambda_{s}\leqslant \lambda_{\bar{s}}$ and $u_{s}\in \mathcal{M}_{s}$,   owing to Lemmata~\ref{lemma lambda1} and~\ref{lemma 2 norm is bound}, we see that 
\begin{equation}\label{vdfv}
	\begin{split}
		\lambda_{\bar{s}}-\lambda_{s}&\leqslant \|u_{s}\|^{-2}_{L^{p+2}(\mathbb{R}^n)}\left(\|u_{s}\|^2_{\bar{s}}-\|u_{s}\|^2_{s}\right) \leqslant \|u_{s}\|^{-2}_{L^{p+2}(\mathbb{R}^n)}\int_{\mathbb{R}^{n}}\left||\xi|^{2\bar{s}}-|\xi|^{2s}\right||\widehat{u_{s}}|^2\\
		&\leqslant 4|\bar{s}-s|\, \|u_{s}\|^{-2}_{L^{p+2}(\mathbb{R}^n)}\int_{\mathbb{R}^{n}}\left((2es)^{-1}+(2\bar{s})^{-1}\right)(1+|\xi|^{4\bar{s}})|\widehat{u_{s}}|^2\\
		&\leqslant 4d^{-1}\lambda_{s}^{-2/p}|\bar{s}-s|\, \|u_s\|^2_{2}\leqslant C(p,d,\lambda_1, n)|\bar{s}-s|.
	\end{split}
\end{equation}
Since the roles
of $s$ and $\bar{s}$ may be interchanged, one derives that
\begin{equation}\label{vsdfvsdfds}
	|\lambda_{\bar{s}}-\lambda_{s}|\leqslant C(p,d,\lambda_1, n)|\bar{s}-s|.
\end{equation}
Therefore, taking $k$ sufficiently large such that $s_k>\gamma/2$, 
\eqref{vsdfvsdfds} implies that $\lambda_{s_k}\rightarrow \lambda_\gamma$ as $s_k\rightarrow \gamma.$ Furthermore, from~\eqref{sdgdf}, we observe that 
$ \|u_{s_k}\|_{L^{p+2}(\mathbb{R}^n)}\rightarrow \|u_{\gamma}\|_{L^{p+2}(\mathbb{R}^n)} $.
Thus, by combining the definition of $J_\gamma$ and the fact that $\lambda_{s_k}=\|u_{s_k}\|^p_{L^{p+2}(\mathbb{R}^n)}$, we obtain that 
\begin{equation}
	\lambda_{\gamma}=\lim\limits_{k\rightarrow\infty}\|u_{s_k}\|^p_{L^{p+2}(\mathbb{R}^n)}= \|u_{\gamma}\|^p_{L^{p+2}(\mathbb{R}^n)} ={\|u_\gamma\|_{\gamma}^2}\|u_\gamma\|^{-2}_{L^{p+2}(\mathbb{R}^n)}=J_\gamma(u_\gamma).
\end{equation}
This implies that $u_\gamma\in\mathcal{M}_\gamma.$

We next consider the case $\gamma=0 $,
 Let $s>0$ sufficiently small, 
we might as well assume that $ \lambda_{s}\leqslant \lambda_{0}$ and $u_{s}\in \mathcal{M}_{s}$,   owing to Lemma~\ref{lemma 2 norm is bound}, we see that 
\begin{equation}\label{vdfvw}
	\begin{split}
		\lambda_{0}-\lambda_{s}&\leqslant \|u_{s}\|^{-2}_{L^{p+2}(\mathbb{R}^n)}\left(\|u_{s}\|^2_{0}-\|u_{s}\|^2_{s}\right) \leqslant \|u_{s}\|^{-2}_{L^{p+2}(\mathbb{R}^n)}\int_{\mathbb{R}^{n}}\left|1-|\xi|^{2s}\right||\widehat{u_{s}}|^2\\
		&\leqslant  2s\|u_{s}\|^{-2}_{L^{p+2}(\mathbb{R}^n)}\left(\|u_s\|^2_{2}+2e^{-1}c_n\|{u}_s\|^2_{L^1(\mathbb{R}^n)}\right)\leqslant s\,C(p,\lambda_1, n,\lambda_0),
	\end{split}
\end{equation}
thanks to~\eqref{bjk} and~\eqref{fbe}.
Thus, one derives that
\begin{equation}\label{vsdfvsdfdsn}
	|\lambda_{0}-\lambda_{s}|\leqslant s\,  C(p,\lambda_0,\lambda_1, n).
\end{equation}
Therefore, 
\eqref{vsdfvsdfdsn} yields that $\lambda_{s_k}\rightarrow \lambda_0$ as $k\rightarrow \infty.$ Furthermore, from~\eqref{sdgdf} and the definition of $J_0$, we obtain that 
\begin{equation}
	\lambda_{0}=\lim\limits_{k\rightarrow\infty}\|u_{s_k}\|^p_{L^{p+2}(\mathbb{R}^n)}= \|u_{0}\|^p_{L^{p+2}(\mathbb{R}^n)} =J_0(u_0).
\end{equation}
This implies that $u_0\in\mathcal{M}_0.$

\smallskip

Step 4. We prove that 	$\|u_{s_k}-u_\gamma\|_{2}\rightarrow 0$  {as } $k\rightarrow \infty$.

To this end, let us denote $w_{s_k}:=u_\gamma-u_{s_k}$, and one has that 
\begin{equation}\label{sdgr}
	-\Delta w_{s_k}+(-\Delta)^{\gamma}w_{s_k}+w_{s_k}=u_\gamma^{p+1}-(u_\gamma-w_{s_k})^{p+1}+(-\Delta)^{s_k}w_{s_k}-(-\Delta)^{\gamma}w_{s_k}.
\end{equation}
From Theorem~\ref{th: existence and properity}, it follows that  $u_\gamma,u_{s_k}\in C^2(\mathbb{R}^n)$, thus, \eqref{sdgr} holds pointwise.
By combining the fundamental theorem of calculus and~\eqref{fbds}, we see that 
\begin{equation}
	\begin{split}
		|u_\gamma^{p+1}-(u_\gamma-w_{s_k})^{p+1}|&\leqslant(p+1)|w_{s_k}| \int_{0}^{1}|u_\gamma+(t-1)w_{s_k}|^p\, dt\\
		&\leqslant (p+1)|w_{s_k}|\left(2\|u_\gamma\|_{L^\infty}+\|u_{s_k}\|_{L^\infty}\right)^p\leqslant C_{n,p,\lambda_1}|w_{s_k}|.
	\end{split}
\end{equation}
As a consequence of this, recalling~\eqref{sdgr}, using again \cite[Lemma~2.4]{MR3207007} for $\gamma>0$ and Lemma~\ref{Lemma gamma=0} for $\gamma=0$, we have that
\begin{equation}
	\begin{split}
	\|w_{s_k}\|^2_{2}&=\int_{{\mathbb{R}^n}}\left(1+|\xi|^{4}+|\xi|^2\right)|\widehat{w_{s_k}}|^2\, d\xi\leqslant\|-\Delta w_{s_k}+(-\Delta)^{\gamma}w_{s_k}+w_{s_k}\|_{L^2(\mathbb{R}^n)}^2\\
	&\leqslant C_{n,p,\lambda_1,\lambda_0,\gamma}\left(\|w_{s_k}\|^2_{L^2(\mathbb{R}^n)}+|s_k-\gamma|^2\right),
	\end{split}
\end{equation}
thanks to Lemma~\ref{lemma 2 norm is bound}. 
From this and~\eqref{cdsvfd}, we conclude that $\|w_{s_k}\|_{2}\rightarrow 0$ as $s_k\rightarrow \gamma$ as desired.
\end{proof}

\section{Nondegeneracy} \label{sec:Nondegeneracy}

This section is devoted  to establish the nondegeneracy result of   $u_s\in\mathcal{M}_s  $ as stated in Theorem~\ref{th:nondegeneracy} by using the uniqueness and nondegeneracy results for the local case.  Namely, for $j\in\left\{0,1\right\}$, there exists a unique radial minimizer $u_j\in \mathcal{M}_j$, such that
\begin{equation}\label{L1}
\text{Ker}\,L_j=\text{span} \left\{\partial_i u_1,\ i=1,\dots,n\right\}.
\end{equation} 

The proof of Theorem~\ref{th:nondegeneracy} will be divided  into two main steps: nondegeneracy both in the radial sector and nonradial sector. First we conclude the triviality of the kernel of linearized operator $L_s$ in the space of radial function.

\subsection{Nondegeneracy in the radial sector}\label{sec:Nondegeneracy in the radial sector}

\begin{Lemma}\label{lemma nondegeneracy of radial sector}
Let $n\geqslant 1 $,	there exists a constant $s_1$ such that for every $s\in(s_1,1)$, we have 
	\begin{equation}
	(\text{Ker}\,L_s)\cap L^2_{rad}(\mathbb{R}^n)=\left\{0\right\}.
	\end{equation}
\end{Lemma}
\begin{proof}
	Let $v_s\in (\text{Ker}\,L_s)\cap L^2_{rad}(\mathbb{R}^n)$. 
	Suppose by contradiction that there exists a sequence $\left\{s_k\right\}$ with $ s_k\nearrow 1$ and $v_{s_k}\neq 0$ satisfying
\begin{equation}\label{fdfgre}
	-\Delta v_{s_k} +(-\Delta)^{s_k} v_{s_k}+v_{s_k}=(p+1)u_{s_k}^pv_{s_k}
\end{equation}
 where $u_{s_k}\in \mathcal{M}_{s_k}$. Up to normalization, we may assume that $\|v_{s_k}\|_{L^{p+2}(\mathbb{R}^n)}=1$.
	From Lemma~\ref{lemma convergence}, it follows that there exists  a unique  $u_1\in\mathcal{M}_1$ such that $u_{s_k}\rightarrow u_1$ in $L^{p+2}(\mathbb{R}^n)$. Moreover, we observe that 
	\begin{equation}\label{}
		\|v_{s_k}\|^2_{s_k}=(p+1)\int_{\mathbb{R}^{n}}|u_{s_k}|^p\,|v_{s_k}|^2\, dx\leqslant(p+1)\|u_{s_k}\|^p_{L^{p+2}(\mathbb{R}^n)}\|v_{s_k}\|^2_{L^{p+2}(\mathbb{R}^n)},
	\end{equation}
and thus, Lemma~\ref{lemma lambda1} implies that $\left\{v_{s_k}\right\}$ is a radial and uniformly bounded in $H^1(\mathbb{R}^n)$ with respect to $s_k$. Hence,  there exists $ v$ such that 
$v_{s_k}\rightharpoonup v$ weakly in $H^1(\mathbb{R}^n)$. Furthermore, by local Rellich compactness, we have that $v_{s_k}\rightarrow v$ strongly in $L^{2}_{\rm loc}(\mathbb{R}^n)$.

 We now claim that 
\begin{equation}\label{cghv}
	v_{s_k}\rightarrow v \quad \text{in } L^2(\mathbb{R}^n).
\end{equation}
Indeed, for given $R>0$, we choose a  cutoff function $\varphi\in C^\infty(\mathbb{R}^n, \mathbb{R})$ satisfying 
\begin{equation}\label{varphi cut off}
	\begin{cases}
	\varphi(x)=0  &\text{ if }  |x|<R\\
		\varphi(x)=1  &\text{ if }  |x|>2R\\
		0\leqslant\varphi \leqslant 1\quad \text{and }\quad  |\nabla\varphi|\leqslant \frac{2}{R}  &\text{ in } \mathbb{R}^n.
	\end{cases}
\end{equation}
By plugging the test function  $\varphi^2 v_{s_k}$ into equation~\eqref{fdfgre}, one has that,
\begin{equation}\label{cdvsdv}
	\begin{split}
	&\int_{{\mathbb{R}^n}}|\nabla(\varphi v_{s_k})|^2+c_{n,s_k}\int_{{\mathbb{R}^{2n}}}\frac{|(\varphi v_{s_k})(x)-(\varphi v_{s_k})(y)|^2}{|x-y|^{n+2s_k}}+\int_{{\mathbb{R}^n}}(\varphi v_{s_k})^2\\
	=&\int_{{\mathbb{R}^n}}v_{s_k}^2|\nabla\varphi|^2+c_{n,s_k}\int_{{\mathbb{R}^{2n}}}\frac{v_{s_k}(x)v_{s_k}(y)|\varphi(x)-\varphi(y)|^2}{|x-y|^{n+2s_k}}+\int_{{\mathbb{R}^n}}(p+1)|u_{s_k}|^p(\varphi v_{s_k})^2\\
	\leqslant& 4R^{-2} \|v_{s_k}\|^2_{L^2(\mathbb{R}^n)}+\frac{c_{n,s_k}}{2}\int_{{\mathbb{R}^{2n}}}\frac{(v_{s_k}^2(x)+v_{s_k}^2(y))|\varphi(x)-\varphi(y)|^2}{|x-y|^{n+2s_k}}+\int_{{\mathbb{R}^n}}(p+1)|u_{s_k}|^p(\varphi v_{s_k})^2\\
	\leqslant & 4R^{-2} \|v_{s_k}\|^2_{L^2(\mathbb{R}^n)}+c_{n,s_k}\int_{{\mathbb{R}^{2n}}}\frac{v_{s_k}^2(x)|\varphi(x)-\varphi(y)|^2}{|x-y|^{n+2s_k}}+\int_{|x|>R}(p+1)|u_{s_k}|^pv_{s_k}^2.
	\end{split}
\end{equation}
From \cite[Lemma 2.5]{MR4670033}, it follows that 
\begin{equation}\label{fbsgsf}
	\int_{{\mathbb{R}^{2n}}}\frac{v_{s_k}^2(x)|\varphi(x)-\varphi(y)|^2}{|x-y|^{n+2s_k}}\leqslant C_n R^{-2s_k} \left(\frac{1}{1-s_k}+\frac{1}{s_k}\right)\|v_{s_k}\|^2_{L^2(\mathbb{R}^n)}.
\end{equation} 
Recalling~\eqref{c_n,s}, by combining~\eqref{cdvsdv} with~\eqref{fbsgsf}, one has that
\begin{equation}\label{svddf}
	\int_{|x|>2R}|v_{s_k}|^2\leqslant \int_{{\mathbb{R}^n}}(\varphi v_{s_k})^2\leqslant C_n\left(R^{-2}+ R^{-2s_k}\right) \|v_{s_k}\|^2_{L^2(\mathbb{R}^n)}+\int_{|x|>R}(p+1)|u_{s_k}|^pv_{s_k}^2.
\end{equation}
Since $v_{s_k}$ is  uniformly bounded in $H^1(\mathbb{R}^n)$, owing to~\eqref{dgsdbf} and the fact that $s_k\nearrow 1$, for any $\epsilon>0$, there exists $R_\epsilon>0$ sufficiently large such that \begin{equation}\label{vsbsf}
	\int_{|x|>R_\epsilon}(p+1)|u_{s_k}|^pv_{s_k}^2 \leqslant C_{n,p,\lambda_1,\lambda_0}|R_\epsilon|^{-{np}}\|v_{s_k}\|^2_{L^2(\mathbb{R}^n)} <\frac{\epsilon^2}{8},
\end{equation}
\begin{equation}\label{sfvdsf}
\text{and }\qquad \qquad	C_n\left(R_\epsilon^{-2}+ R_\epsilon^{-2s_k}\right) \|v_{s_k}\|^2_{L^2(\mathbb{R}^n)}<\frac{\epsilon^2}{8}.
\end{equation}
Moreover, picking $R:=R_\epsilon$ in~\eqref{varphi cut off},  by combining~\eqref{svddf}, ~\eqref{vsbsf} and~\eqref{sfvdsf}, one has that 
$\int_{|x|>2R_\epsilon}|v_{s_k}|^2\leqslant \epsilon^2/4.$ Furthermore, the Fatou Lemma implies that $	\int_{|x|>2R_\epsilon}|v|^2\leqslant \epsilon^2/4.$

Since $v_{s_k}\rightarrow v$ in $L^{2}_{\rm loc}(\mathbb{R}^n)$, there exists $k_0>0$ such that 
\begin{equation}\label{vdsvds}
	\int_{|x|\leqslant 2R_\epsilon} |v_{s_k}-v|^{2}\, dx\leqslant\frac{\epsilon^{2}}{2}\quad \text{for every } k>k_0.
\end{equation}
Gathering these above facts,  we thus conclude that
\begin{equation}\label{vvdsvds}
	\|v_{s_k}-v\|_{L^{2}(\mathbb{R}^n)}\leqslant \epsilon \quad \text{for all } k> k_0.
\end{equation}
 This implies the claim~\eqref{cghv}.

Moreover, the Sobolev embeddings and uniform bound on $H^1$-norm of $v_{s_k}$ ensure that $\|v_{s_k}\|_{L^q(\mathbb{R}^n)}\leqslant C_{n,p,\lambda_1,\lambda_0}$  where $q$ is either equal to  $\frac{2n}{n-2}$ if $n > 2$ or any real number in $(p+2,+\infty)$
if $n =1,2$. Thus, based on the strong convergence of  $\left\{v_{s_k}\right\}$ in $L^{2}(\mathbb{R}^n)$, by interpolation inequality and Fatou Lemma, we deduce that $v_{s_k}\rightarrow v$ in $L^{p+2}(\mathbb{R}^n)$.

In particular, $\|v\|_{L^{p+2}(\mathbb{R}^n)}=1$. Moreover, since $v_{s_k}$ is a solution of~\eqref{fdfgre}, one has that for every $\varphi\in C_c^\infty(\mathbb{R}^n)$
\begin{equation}
	\int_{\mathbb{R}^{n}}-\Delta \varphi \,v_{s_k}+	\int_{\mathbb{R}^{n}}(-\Delta)^{s_k} \varphi \,v_{s_k}+	\int_{\mathbb{R}^{n}} \varphi \,v_{s_k}=(p+1)	\int_{\mathbb{R}^{n}}u_{s_k}^p \varphi \,v_{s_k}.
\end{equation}
  In view of \cite[Lemma~2.4]{MR3207007}, we know that $(-\Delta)^{s_k}\varphi\rightarrow -\Delta\varphi$ in $L^2(\mathbb{R}^n)$. Therefore, by combining~\eqref{cdsvfd} with ~\eqref{cghv}, we infer that 
  \begin{equation}
  	\int_{\mathbb{R}^{n}}-\Delta \varphi \,v+	\int_{\mathbb{R}^{n}} -\Delta \varphi \,v+	\int_{\mathbb{R}^{n}} \varphi \,v=(p+1)	\int_{\mathbb{R}^{n}}u_{1}^p \varphi \,v.
  \end{equation}
Therefore, we conclude that $v\in \text{Ker}\,L_1\cap L^2_{\rm rad}(\mathbb{R}^n)$. This  clearly contradicts~\eqref{L1}.
\end{proof}

\begin{Lemma}\label{lemma nondegeneracy of radial sector s1}
	Let $n\geqslant 1 $,	there exists a constant $s_0$ such that for every $s\in(0,s_0)$, we have 
	\begin{equation}
		(\text{Ker}\,L_s)\cap L^2_{rad}(\mathbb{R}^n)=\left\{0\right\}.
	\end{equation}
\end{Lemma}
\begin{proof}
	Let $w_s\in (\text{Ker}\,L_s)\cap L^2_{rad}(\mathbb{R}^n)$. 
We argue by contradiction and we suppose  that there exists a sequence $\left\{s_k\right\}$ with $ s_k\searrow 0$ and $w_{s_k}\neq 0$ satisfying
\begin{equation}\label{fdfgreq}
	-\Delta w_{s_k} +(-\Delta)^{s_k} w_{s_k}+w_{s_k}=(p+1)u_{s_k}^pw_{s_k}
\end{equation}
where $u_{s_k}\in \mathcal{M}_{s_k}$. Up to normalization, we may assume that $\|w_{s_k}\|_{L^{p+2}(\mathbb{R}^n)}=1$.
From Lemma~\ref{lemma convergence}, it follows that there exists  a unique  $u_0\in\mathcal{M}_0$ such that $u_{s_k}\rightarrow u_0$ in $L^{p+2}(\mathbb{R}^n)$. Moreover, we observe that 
\begin{equation}
	\|w_{s_k}\|^2_{s_k}=(p+1)\int_{\mathbb{R}^{n}}|u_{s_k}|^p\,|w_{s_k}|^2\, dx\leqslant(p+1)\|u_{s_k}\|^p_{L^{p+2}(\mathbb{R}^n)}\|w_{s_k}\|^2_{L^{p+2}(\mathbb{R}^n)},
\end{equation}
and thus, Lemma~\ref{lemma lambda1} implies that $\left\{w_{s_k}\right\}$ is a radial and uniformly bounded in $H^1(\mathbb{R}^n)$ with respect to $s_k$. Hence,  there exists $ w$ such that 
$w_{s_k}\rightharpoonup w$ weakly in $H^1(\mathbb{R}^n)$.

We now claim that 
\begin{equation}\label{cghnv}
	w_{s_k}\rightarrow w \quad \text{in } L^2(\mathbb{R}^n).
\end{equation}
Indeed, we take $V(x)=-(p+1)u_{s_k}^p(x)$, $W=-1$ and $\lambda=\frac{1}{2}$ in Lemma~\ref{lemma uniform decay of spectrum}, owing to~\eqref{dgsdbf}, obtaining that  there exists $R>0$ independent of $s$ such that 
\begin{equation}\label{ver}
	V(x)+\frac{1}{2}>0 \qquad \text{for every } |x|\geqslant R.
\end{equation} 
Thus, Lemma~\ref{lemma uniform decay of spectrum} implies that $\|w_{s_k}\|_{L^\infty(\mathbb{R}^n)}\leqslant C$  which is independent of $k$, and  there exists a positive constant $s_1$ such that, for all $|x|\geqslant 1$ and $s_k\in(0,s_1]$ 
\begin{equation}\label{sdvbsdnkm}
	|w_{s_k}(x)|\leqslant C\left(n,R,p,\|u_{s_k}\|_{L^\infty(\mathbb{R}^n)}\right)|x|^{-n}.
\end{equation}
From this and Lemma~\ref{lemma uniform bound}, we conclude that for any $\epsilon>0$, owing to Fatou Lemma, there exists  $R_\epsilon>0$ independent of $s$ such that 
\begin{equation}\label{dsvsd}
	\int_{|x|>R_\epsilon} |w_{s_k}(x)|^2\, dx\leqslant \epsilon^2/4\quad \text{and} \quad \int_{|x|>R_\epsilon} |w(x)|^2\, dx\leqslant \epsilon^2/4.
\end{equation}
 Furthermore, by local Rellich compactness, we have that $w_{s_k}\rightarrow w$ strongly in $L^{2}_{\rm loc}(\mathbb{R}^n)$. From this and~\eqref{dsvsd},  we obtain the desired result~\eqref{cghnv}.
 
 Moreover,  based on the strong convergence of  $\left\{w_{s_k}\right\}$ in $L^{2}(\mathbb{R}^n)$, by combining the Sobolev embeddings and the interpolation inequality, we deduce that $w_{s_k}\rightarrow w$ in $L^{p+2}(\mathbb{R}^n)$.

 In particular, $\|w\|_{L^{p+2}(\mathbb{R}^n)}=1$. Moreover, since $w_{s_k}$ is a solution of~\eqref{fdfgreq}, one has that for every $\varphi\in C_c^\infty(\mathbb{R}^n)$
 \begin{equation}
 	\int_{\mathbb{R}^{n}}-\Delta \varphi \,w_{s_k}+	\int_{\mathbb{R}^{n}}(-\Delta)^{s_k} \varphi \,w_{s_k}+	\int_{\mathbb{R}^{n}} \varphi \,w_{s_k}=(p+1)	\int_{\mathbb{R}^{n}}u_{s_k}^p \varphi \,w_{s_k}.
 \end{equation}
 In view of Lemma~\ref{Lemma gamma=0}, we know that $(-\Delta)^{s_k}\varphi\rightarrow \varphi$ in $L^2(\mathbb{R}^n)$. Therefore, by combining~\eqref{cdsvfd} with ~\eqref{cghnv}, we infer that 
 \begin{equation}
 	\int_{\mathbb{R}^{n}}-\Delta \varphi \,w+	\int_{\mathbb{R}^{n}} \varphi \,w+	\int_{\mathbb{R}^{n}} \varphi \,w=(p+1)	\int_{\mathbb{R}^{n}}u_{0}^p \varphi \,w.
 \end{equation}
 Therefore, we conclude that $w\in \text{Ker}\,L_0\cap L^2_{\rm rad}(\mathbb{R}^n)$. This  clearly contradicts~\eqref{L1}.
\end{proof}

\subsection{Nondegeneracy in the nonradial sector}\label{sec:Nondegeneracy in the nonradial sector}
From Lemmata~\ref{lemma nondegeneracy of radial sector} and~\ref{lemma nondegeneracy of radial sector s1},  we see that the triviality of the kernel of the linearized operator $L_s$ in the space of radial functions when $s$ close to $1$ and $0$. Furthermore,  to rule out further elements in the kernel of $L_s$ apart from $\partial_i u_s$ with $i=1,\dots,n$, the strategies are: 
\begin{itemize}
\item when $n=1$, we use heat kernel and Perron-Frobenius arguments in Subsection~\ref{sec n=1}, 
\item when $n\geqslant 2$, we combine  a local realization  of  $(-\Delta)^s $ with  spherical harmonics  decomposition method in Subsections~\ref{sec local realization}-\ref{spherical harmonics}.
\end{itemize}

\subsubsection{Nondegeneracy for $n=1$}\label{sec n=1}

Let $u_s\in \mathcal{M}_s$. Denote 
$$L^2_{even}(\mathbb{R}):=\left\{f\in L^2(\mathbb{R}): f(x)=f(-x) \text{ for a.e.} x\in \mathbb{R}  \right\} $$ $$ \text{and }\quad  L^2_{odd}(\mathbb{R}):=\left\{f\in L^2(\mathbb{R}): f(x)=-f(-x) \text{ for  a.e.} x\in \mathbb{R} \right\}. $$
 We now consider the orthogonal decomposition $L^2(\mathbb{R})=L^2_{even}(\mathbb{R})\bigoplus L^2_{odd}(\mathbb{R})$. Since
$ u_s=u_s(|x|)$ is an even function. Furthermore, we notice that $L_s$ leaves the subspaces $L^2_{even}(\mathbb{R})$ and $L^2_{odd}(\mathbb{R}) $
 invariant.
 
\begin{proof}[Proof of Theorem~\ref{th:nondegeneracy} $(n=1)$]
We observe that $\partial_{r}u_s\in L^2_{odd}(\mathbb{R})$ and $\partial_{r}u_s\leqslant 0 $ for $x>0$.  Additionally, $\partial_r u_s\in \text{Ker}\,L_s$, namely  $$-\Delta\partial_{r}u_s+(-\Delta)^s \partial_{r}u_s-(p+1)u_s^p\partial_{r}u_s+\partial_{r}u_s=0.$$
From Proposition~\ref{proposition A.1}, we know that $L_s$ has exactly one negative eigenvalue and the corresponding eigenfunction is even. Thus,  in view of Lemma~\ref{lemma n=1} applied to $L_s$, owing to  Remark~\ref{reamrk 1.2}, we deduce that  $\partial_r u_s\in L^2_{odd}(\mathbb{R})$ is the 
eigenfunction of the lowest eigenvalue of $L_s$ restricted to $L^2_{odd}(\mathbb{R})$. Thus, $\partial_{r}u_s< 0 $ for $x>0$ and $\partial_{r}u_s$ is (up to
a sign) the unique eigenfunction belonging to   $L^2_{odd}(\mathbb{R})\cap \text{Ker}\,L_s $. From this, by combining Lemmata~\ref{lemma nondegeneracy of radial sector} and~\ref{lemma nondegeneracy of radial sector s1}, we complete the proof of Theorem~\ref{th:nondegeneracy} for space dimension $n=1$.  
\end{proof}

We next discuss the nondegeneracy of ground states in the nonradial sector for space dimension $n\geqslant 2$.

\subsubsection{Local realization of $(-\Delta)^s $ with $s\in(0,1)$}\label{sec local realization}
We start by recalling the extension principle in \cite{MR2354493} that expresses the nonlocal operator $(-\Delta)^s$ on $\mathbb{R}^n$ with $s\in(0,1)$ as a generalized Dirichlet-Neumann map for a suitable local elliptic problem posed on the upper half-space $\mathbb{R}^{n+1}_+=\left\{(x,t):\ x\in\mathbb{R}^n,\ t>0\right\}$. That is, given a solution $u$ of $(-\Delta)^s u=f$ in $\mathbb{R}^n$, we can equivalently consider the dimensionally extended problem for $U=U(x,t)$, which solves
\begin{equation}
	\begin{cases}\label{localization}
		\text{div }(t^{1-2s}\nabla U)=0\qquad &\text{ in } \mathbb{R}^{n+1}_+\\
		U(x,0)=u &\text{ in } \mathbb{R}^{n}\\
		-d_s t^{1-2s}\partial_t U|_{t\rightarrow 0}=f
&\text{ on }  \mathbb{R}^{n}.
	\end{cases}
\end{equation}
Here the positive constant $d_s=2^{2s-1}\frac{\Gamma(s)}{\Gamma(1-s)}$.

%
   Moreover, based on \eqref{localization} and  the mixed local/nonlocal nature of the equation~\eqref{main equation}, we  consider the following Sobolev space:
\begin{equation}\label{working space}
	\mathscr{H}_0^1(\mathbb{R}^{n+1}_+;\, t^{1-2s}):=\left\{U\in H^1(\mathbb{R}^{n+1}_+; \,t^{1-2s}):\ \text{Tr}U\in H^1(\mathbb{R}^n)\right\},
\end{equation}
equipped with the norm 
\[ \|U\|^2_{\mathscr{H}_0^1}=\int_{\mathbb{R}^{n+1}_+}\left(|U|^2+|\nabla U|^2\right)t^{1-2s}\,dt dx+\int_{\mathbb{R}^{n}}\left(| \text{Tr}U|^2+|\nabla \text{Tr}U|^2\right)\, dx. \]
Here Tr$U $ is the trace  of function $U$ on $\mathbb{R}^n$. From \cite[Theorem 18.57]{MR3726909}, we know that  for every $U\in H^1(\mathbb{R}^{n+1}_+;\, t^{1-2s}) $, Tr$U\in H^s(\mathbb{R}^n)$. Thus, the space $\mathscr{H}_0^1(\mathbb{R}^{n+1}_+;\, t^{1-2s})$ defined in~\eqref{working space} is a complete subspace of the weighted Sobolev space $H^1(\mathbb{R}^{n+1}_+;\,t^{1-2s})$.

As a result, given $u\in H^1(\mathbb{R}^n)$, there exists a unique extension $U\in \mathscr{H}_0^1(\mathbb{R}^{n+1}_+;\, t^{1-2s})$ such that
\begin{equation}
	\begin{cases}\label{local realization}
		\text{div }(t^{1-2s}\nabla U)=0\qquad &\text{ in } \mathbb{R}^{n+1}_+\\
		U(x,0)=u &\text{ in } \mathbb{R}^{n}\\
		-\Delta u-d_s t^{1-2s}\partial_t U|_{t\rightarrow 0}=	-\Delta u+(-\Delta)^su
		&\text{ on }  \mathbb{R}^{n}.
	\end{cases}
\end{equation}
Equivalently, for every $\Psi\in \mathscr{H}_0^1(\mathbb{R}^{n+1}_+;\, t^{1-2s})$, let us denote $\psi:=$Tr$\Psi$. Then,
\begin{equation}\label{cnvisdhvv}
\int_{{\mathbb{R}^n}}\nabla u\cdot\nabla\psi\, dx+	\int_{\mathbb{R}^{n+1}_+}\nabla U\cdot\nabla\Psi\, t^{1-2s}\,dtdx=\int_{{\mathbb{R}^n}}|\xi|^{2}\hat{u}\hat{\psi}\, d\xi+d_s^{-1}\int_{{\mathbb{R}^n}}|\xi|^{2s}\hat{u}\hat{\psi}\, d\xi.
\end{equation}
From this, we observe that
\begin{equation}
	\begin{split}
		&\inf\limits_{\Psi\in \mathscr{H}_0^1(\mathbb{R}^{n+1}_+;\, t^{1-2s})\backslash\left\{0\right\}}\frac{d_s\int_{\mathbb{R}^{n+1}_+}|\nabla \Psi|^2t^{1-2s}dxdt+\int_{\mathbb{R}^{n}}\left(1+|\xi|^2\right)|\hat{\psi}|^2\, d\xi}{\|\psi\|^2_{L^{p+2}(\mathbb{R}^n)}}\\
		=&\inf\limits_{\Psi\in {H}^1(\mathbb{R}^{n+1}_+;\, t^{1-2s}) \atop
			 \psi\in H^1(\mathbb{R}^n)\backslash\left\{0\right\}}\frac{d_s\int_{\mathbb{R}^{n+1}_+}|\nabla \Psi|^2t^{1-2s}dxdt+\int_{\mathbb{R}^{n}}\left(1+|\xi|^2\right)|\hat{\psi}|^2\, d\xi}{\|\psi\|^2_{L^{p+2}(\mathbb{R}^n)}}\\
		 =& \inf\limits_{\psi\in H^1(\mathbb{R}^n)\backslash\left\{0\right\}}\frac{\int_{\mathbb{R}^{n}}\left(1+|\xi|^{2s}+|\xi|^2\right)|\hat{\psi}|^2\, d\xi}{\|\psi\|^2_{L^{p+2}(\mathbb{R}^n)}}=\inf\limits_{\psi\in H^1(\mathbb{R}^n)\backslash\left\{0\right\}}J_s(\psi).
	\end{split}
\end{equation}
Hence, let $u_s\in\mathcal{M}_s$, the extension $U_s$ of $u_s$ is radially symmetric with respect to the $x$ variable and it is a minimizer for 
\begin{equation}\label{dfwedf}
	\begin{split}
	 \lambda_s=&\inf\limits_{\Psi\in \mathscr{H}_0^1(\mathbb{R}^{n+1}_+;\, t^{1-2s})\backslash\left\{0\right\}}\frac{d_s\int_{\mathbb{R}^{n+1}_+}|\nabla \Psi|^2t^{1-2s}dxdt+\int_{\mathbb{R}^{n}}\left(1+|\xi|^2\right)|\hat{\psi}|^2\, d\xi}{\|\psi\|^2_{L^{p+2}(\mathbb{R}^n)}}\\
	 =: &\inf\limits_{\Psi\in \mathscr{H}_0^1(\mathbb{R}^{n+1}_+;\, t^{1-2s})\backslash\left\{0\right\}} \mathscr{J}_s(\Psi).
 \end{split}
\end{equation}
As a consequence of this and~\eqref{local realization}, one has that
\begin{equation}
	\begin{cases}\label{local realization2}
		\text{div }(t^{1-2s}\nabla U_s)=0\qquad &\text{ in } \mathbb{R}^{n+1}_+\\
		U_s(x,0)=u_s &\text{ in } \mathbb{R}^{n}\\
		-\Delta u_s-d_s t^{1-2s}\partial_t U_s|_{t\rightarrow 0}+u_s=	u_s^{p+1}
		&\text{ on }  \mathbb{R}^{n}.
	\end{cases}
\end{equation}

Moreover, we have
\begin{Lemma}\label{lemma vsdv}
	Assume that $u_s\in\mathcal{M}_s$, $U_s$ is the extension of $u_s$.
Let $\Psi\in \mathscr{H}_0^1(\mathbb{R}^{n+1}_+;\, t^{1-2s})$ be such that
\begin{equation}\label{bgbjk}
	d_s\int_{\mathbb{R}^{n+1}_+}\nabla U_s\cdot \nabla \Psi t^{1-2s}\, dxdt+\int_{{\mathbb{R}^n}}\nabla u_s\cdot \nabla\psi\,dx+\int_{{\mathbb{R}^n}}u_s\, \psi\, dx=0,
\end{equation}
where $\psi=$Tr$\Psi$. Then,
\begin{equation}\label{bjkj}
	d_s\int_{\mathbb{R}^{n+1}_+}|\nabla \Psi|^2 t^{1-2s}\, dxdt+\int_{{\mathbb{R}^n}} |\nabla\psi|^2\,dx+\int_{{\mathbb{R}^n}} |\psi|^2\, dx-(p+1)\int_{{\mathbb{R}^n}}u_s^p|\psi|^2\, dx\geqslant0.
\end{equation}
\end{Lemma}
\begin{proof}
Since $U_s$ minimizers \eqref{dfwedf}, for any $\epsilon>0$, one has that
\begin{equation}\label{sdgsrb}
	 \mathscr{J}_{s}(U_s+\epsilon\Psi)\geqslant\mathscr{J}_s(U_s).
\end{equation}
Owing to~\eqref{bgbjk},  we see that
\begin{equation}\label{dfwev}
	\begin{split}
		&d_s\int_{\mathbb{R}^{n+1}_+}|\nabla U_s+\epsilon\nabla \Psi|^2t^{1-2s}dxdt+\int_{\mathbb{R}^{n}}|\nabla u_s+\epsilon\nabla \psi|^2\, dx+\int_{{\mathbb{R}^n}}|u_s+\epsilon \psi|^2\, dx\\
		=&d_s\int_{\mathbb{R}^{n+1}_+}|\nabla U_s|^2t^{1-2s}dxdt+\int_{\mathbb{R}^{n}}|\nabla u_s|^2\, dx+\int_{{\mathbb{R}^n}}|u_s|^2\, dx\\
		&\qquad +\epsilon^2\left(d_s\int_{\mathbb{R}^{n+1}_+}|\nabla \Psi|^2t^{1-2s}dxdt+\int_{\mathbb{R}^{n}}|\nabla \psi|^2\, dx+\int_{{\mathbb{R}^n}}|\psi|^2\, dx\right).
	\end{split}
\end{equation}
By combining a Taylor expansion and \eqref{bgbjk},  we deduce that 
\begin{equation}
	\int_{\mathbb{R}^n}|\epsilon \psi+u_s|^{p+2}\, dx=\int_{\mathbb{R}^n}|u_s|^{p+2}\,dx +\frac{\epsilon^2(p+2)(p+1)}{2}\int_{\mathbb{R}^n}|u_s|^{p}\psi^2\, dx+O(\epsilon^3).
\end{equation}
Utilizing again the Taylor expansion, one can obtain that
\begin{equation}\label{dgav}
	\left(\int_{\mathbb{R}^n}|\epsilon \psi+u_s|^{p+2}\, dx\right)^{\frac{2}{p+2}}= \|u_s\|^{2}_{L^{p+2}(\mathbb{R}^n)}\left(1+\epsilon^2(p+1)\frac{\int_{{\mathbb{R}^n}}|u_s|^{p}\psi^2\, dx}{\|u_s\|^{p+2}_{L^{p+2}(\mathbb{R}^n)}}+O(\epsilon^3)\right).
\end{equation}
Recalling the fact that $u_s\in\mathcal{M}_s$, by~\eqref{cnvisdhvv}, we know that $$\|u_s\|^{p+2}_{L^{p+2}(\mathbb{R}^n)}=\|u_s\|_{s}^2=d_s\int_{\mathbb{R}^{n+1}_+}|\nabla U_s|^2t^{1-2s}dxdt+\int_{\mathbb{R}^{n}}|\nabla u_s|^2\, dx+\int_{{\mathbb{R}^n}}|u_s|^2\, dx.$$
Thus,	by inserting \eqref{dfwev} and \eqref{dgav} into \eqref{sdgsrb}, one has that 
\begin{equation}
	\begin{split}
		0&\leqslant	\epsilon^2\bigg(d_s\int_{\mathbb{R}^{n+1}_+}|\nabla \Psi|^2t^{1-2s}dxdt+\int_{\mathbb{R}^{n}}|\nabla \psi|^2\, dx+\int_{{\mathbb{R}^n}}|\psi|^2\, dx\\
		&\qquad \qquad-(p+1)\int_{\mathbb{R}^n}|u_s|^{p}\psi^2\, dx+O(\epsilon)\bigg).
	\end{split}
\end{equation}
This implies the desired result~\eqref{bjkj}.
\end{proof}

\subsubsection{Nondegenercy when $n\geqslant 2$}\label{spherical harmonics}
We now consider the decomposition in terms of spherical harmonics
 \begin{equation}
 	L^2(\mathbb{R}^n)=\bigoplus\limits_{k\geqslant0} \left(L^2(\mathbb{R}_+; r^{n-1})\otimes \mathcal{Y}_k\right).
 \end{equation}
Here  $\mathcal{Y}_k=$span$\left\{Y_{k}^i\right\}_{i\in m_k}$ denotes the space of  spherical harmonics of degree $k$  in space dimension $n$.  Note that the index set $m_k$ depends on $k$ and $n$. It is known that $m_0=1$ and $m_1=n$.
  Recalling also that for $n\geqslant 2$,
 \begin{equation}\label{dvsdv}
 	-\Delta_{S^{n-1}} Y_k^i =\lambda_{k} Y_k^i,
 \end{equation}
 where $\lambda_k=k(k+n-2)$, and 
\begin{equation}\label{ scvdsvcds}
	\int_{S^{n-1}} Y_{k}^i Y_{l}^j=0\qquad \text{for any }(k,i)\neq(l,j).
\end{equation}
  In addition, $Y_0$ is a constant, while 
 \begin{equation}
 	Y_{1}^i=\frac{x_i}{|x|},\qquad \text{ for } i=1,\dots,n.
 \end{equation}

 Let $W\in\mathscr{H}_0^1(\mathbb{R}^{n+1}_+;\, t^{1-2s})$, we decompose $W$ in the spherical harmonics and obtain that 
\begin{equation}\label{bdivsd}
	W(x,t)=\sum_{k\in\mathbb{N}}\sum_{i=1}^{m_k} G_k^i(|x|,t)Y_{k}^i\left(\frac{x}{|x|}\right),
\end{equation}
and 
\begin{equation}\label{fsbbf}
	\int_{S^{n-1}}|W(r\theta,t)|^2\, d\sigma(\theta)=\sum_{k\in\mathbb{N}}\sum_{i=1}^{m_k} \left|G_k^i(r,t)\right|^2.
\end{equation}
Moreover, owing to~\eqref{dvsdv},~\eqref{ scvdsvcds} and~\eqref{bdivsd}, we infer that 
\begin{equation}
	\begin{split}
	\int_{S^{n-1}}|\nabla_{\theta} W(r\theta, t)|^2\, d\sigma(\theta)=&\sum_{k\in\mathbb{N}}\sum_{i=1}^{m_k}\int_{S^{n-1}}|G_k^i(r,t)|^2|\nabla_{\theta}Y_k^i(\theta)|^2\, d\sigma(\theta)\\
		=&\sum_{k\in\mathbb{N}}\sum_{i=1}^{m_k} \lambda_k \left|G_k^i(r,t)\right|^2.
	\end{split}
\end{equation}
From this, we have that 
\begin{equation}\label{dfsdv}
	\begin{split}
		\int_{{\mathbb{R}^{n+1}_+}}t^{1-2s}|\nabla W(x,t)|^2\,dxdt=&\int_{\mathbb{R}^2_{++}}t^{1-2s}r^{n-1}\int_{S^{n-1}}|\partial_{r}W|^2+|\partial_{t} W|^2+r^{-2}|\nabla_{\theta}W|^2\, d\sigma(\theta)\,drdt\\
	=&\sum_{k\in\mathbb{N}}\sum_{i=1}^{m_k}\int_{\mathbb{R}^2_{++}}t^{1-2s} r^{n-1}\left( \left|\nabla G_k^i(r,t)\right|^2+ r^{-2}\lambda_k \left|G_k^i(r,t)\right|^2\right)\, dtdr.
	\end{split}
\end{equation}
It follows that $G_i^k\in H^1(\mathbb{R}^2_{++}; t^{1-2s}r^{n-1})\cap L^2(\mathbb{R}^2_{++}; t^{1-2s}r^{n-3})$ for any $k\in\mathbb{N}$ and $i=1,\dots,m_k$. 

Recalling~\eqref{bdivsd}, by taking into account the fact that  $W(x,0)\in H^1(\mathbb{R}^n)$, we see that the trace $g_k^i$ on $\mathbb{R}_+$ of $G_k^i$ belongs to $ H^1(\mathbb{R}_{+};\, r^{n-1})\cap L^2(\mathbb{R}_{+};\, r^{n-3}) $.

Thus, let us denote  
	\[ \mathcal{G}:=\left\{G\in H^1(\mathbb{R}^2_{++} t^{1-2s}r^{n-1}):\, G\in L^2(\mathbb{R}^2_{++}; t^{1-2s}r^{n-3}),\text{ Tr} G \in H^1(\mathbb{R}_{+};\, r^{n-1})\cap L^2(\mathbb{R}_{+};\, r^{n-3})
	\right\}.  \]

Now we are ready to prove our nondegeneracy result for $s\in(0,1)$.
\begin{proof}[Proof of Theorem~\ref{th:nondegeneracy} $(n\geqslant2)$]
Let $u_s\in\mathcal{M}_s$ and $w\in \text{Ker}\,L_s$, one has that 
\begin{equation}
	-\Delta w+(-\Delta)^s w+w=(p+1)u_s^pw \quad \text{in }\mathbb{R}^n.
\end{equation}
Let $W\in 	\mathscr{H}_0^1(\mathbb{R}^{n+1}_+;\, t^{1-2s})$ be the extension of $w$, which satisfies for all $\Psi\in 	\mathscr{H}_0^1(\mathbb{R}^{n+1}_+;\, t^{1-2s}) $
\begin{equation}\label{fvfav}
	d_s\int_{\mathbb{R}^{n+1}_+}\nabla W \cdot\nabla \Psi t^{1-2s}\, dxdt+\int_{{\mathbb{R}^n}} \nabla w\cdot\nabla\psi\,dx+\int_{{\mathbb{R}^n}} w\psi\, dx=(p+1)\int_{{\mathbb{R}^n}}u_s^pw\psi\, dx
\end{equation}
where $\psi$ is the trace of  $\Psi$.
Using the decomposition in terms of spherical harmonics, one has that 
\begin{equation}\label{jvdfvdf}
	W(x,t)=\sum_{k\in\mathbb{N}}\sum_{i=1}^{m_k} G_k^i(|x|,t)Y_{k}^i\left(\frac{x}{|x|}\right)
\end{equation}
where $G_k^i\in\mathcal{G} $.
Taking $F\in \mathcal{G},$ by testing~\eqref{fvfav} against function $\Psi=F(t,|x|)Y_k^i\in\mathscr{H}_0^1(\mathbb{R}^{n+1}_+;\, t^{1-2s})$ and using polar coordinates, one has that, for any $k\in\mathbb{N}$, and $i=1,\dots,m_k$,
\begin{equation}
	\begin{split}\label{vfvsd}
	I_k(G_k^i,F):&=d_s\int_{\mathbb{R}^{2}_{++}}t^{1-2s}r^{n-1}\nabla G_k^i\cdot\nabla F\, drdt+d_s\lambda_k\int_{\mathbb{R}^{2}_{++}}t^{1-2s}r^{n-3}{G^i_k F}\, drdt\\
	&\quad +\int_{{\mathbb{R}_+}}r^{n-1}\partial_r g_{k}^i\cdot\partial_rf\, dr+\lambda_k\int_{{\mathbb{R}_+}}r^{n-3} g_{k}^if\, dr+\int_{{\mathbb{R}_+}}r^{n-1} g_{k}^if\, dr\\
	&\quad -(p+1)\int_{{\mathbb{R}_+}}r^{n-1} u_s^pg_{k}^if\, dr=0
	\end{split}
\end{equation}
where $f$ is the trace of $F$ on $\mathbb{R}_+$.

Moreover, since $U_s$ is radial in $x$ variable, we notice that for any $G\in\mathcal{G}$ and $i=1,\dots,n$, $\Phi(x):=G(|x|,t)Y_1^i$ satisfies~\eqref{bgbjk} by odd symmetry. Hence, Lemma~\ref{lemma vsdv} implies that 
	\begin{equation}\label{vdsdfsd}
	\begin{split}
		I_1(G,G):&=d_s\int_{\mathbb{R}^{2}_{++}}\left|\nabla G\right|^2t^{1-2s}r^{n-1}drdt+(n-1)d_s\int_{\mathbb{R}^{2}_{++}}G^2t^{1-2s}r^{n-3}drdt\\
		&\qquad +\int_{{\mathbb{R}_+}}\partial^2_{r}gr^{n-1}\, dr+(n-1)\int_{{\mathbb{R}_+}}g^2r^{n-3}\, dr+\int_{{\mathbb{R}_+}}g^2r^{n-1}\,  dr\\
		&\qquad -(p+1)\int_{{\mathbb{R}_+}}u_s^{p}g^2r^{n-1}\,dr\geqslant 0.
	\end{split}
\end{equation}
Here $g:=$Tr $G$. Recalling the fact that $\lambda_k>n-1$ for $k\geqslant 2$, by combining~\eqref{vfvsd} with~\eqref{vdsdfsd},  we obtain that 
\begin{equation}
	\begin{split}
	0=	I_k(G_k^i,G_k^i)&=I_1(G_k^i,G_k^i)+(\lambda_k-(n-1))d_s\int_{\mathbb{R}^{2}_{++}}t^{1-2s}r^{n-3}{|G^i_k|^2}\, drdt\\
		& \qquad +(\lambda_k-(n-1))\int_{{\mathbb{R}_+}}r^{n-3} |g_{k}^i|^2\, dr\geqslant 0.
	\end{split}
\end{equation}
As a consequence, $G_k^i=0$ and $g_k^i=0$ for every $k\geqslant 2$ and $i=1,\dots,m_k$. From this, \eqref{jvdfvdf} can be rewritten as
\begin{equation}\label{vdsvdsdv}
	W(x,t)=G_0^1(|x|,t)+\sum_{i=1}^{n} G_1^i(|x|,t)Y_{1}^i\left(\frac{x}{|x|}\right).
\end{equation}
Also, we notice that for any $i=1,\dots,n$, 
\begin{equation}
	G_1^i(r,t)=\int_{S^{n-1}}W(r\theta,t)Y_{1}^i(\theta)\, d\sigma(\theta).
\end{equation}
By \eqref{vfvsd}, for every $F\in\mathcal{G}$, one has that
\begin{equation}
	\begin{split}\label{sdvsdv}
		I_1(G_1^i,F):&=d_s\int_{\mathbb{R}^{2}_{++}}t^{1-2s}r^{n-1}\nabla G_1^i\cdot\nabla F\, drdt+d_s(n-1)\int_{\mathbb{R}^{2}_{++}}t^{1-2s}r^{n-3}{G^i_1 F}\, drdt\\
		&\quad +\int_{{\mathbb{R}_+}}r^{n-1}\partial_r g_{1}^i\cdot\partial_rf\, dr+(n-1)\int_{{\mathbb{R}_+}}r^{n-3} g_{1}^if\, dr+\int_{{\mathbb{R}_+}}r^{n-1} g_{1}^if\, dr\\
		&\quad -(p+1)\int_{{\mathbb{R}_+}}r^{n-1} u_s^pg_{1}^if\, dr=0.
	\end{split}
\end{equation}

Moreover, since $U_s$ is radial in $x$ variable, we denote $\bar{U}_s(|x|,t)=U_s(x,t)$ and $\bar{u}_s(|x|)=u_s(x)$,  then we have that 
\begin{equation}
	\begin{cases}
		\text{div }(t^{1-2s}r^{n-1}\nabla \bar{U}_s)=0\quad &\text{ in } \mathbb{R}^{2}_{++}\\
		 \left(-\partial_{rr}\bar{u}_s-\frac{n-1}{r}\partial_r\bar{u}_s\right)r^{n-1}-d_s t^{1-2s}r^{n-1}\partial_t \bar{U}_s|_{t\rightarrow 0}+r^{n-1}\bar{u}_s=	r^{n-1}\bar{u}_s^{p+1}
		&\text{ on }  \mathbb{R}_+.
	\end{cases}
\end{equation}
 We now differentiate the above equation with respect to $r$ and denote $V_s:=\partial_{r}\bar{U}_s$ and $v_s:=\partial_{r}\bar{u}_s$. Then, one has that 
\begin{equation}
	\begin{cases}\label{vsdf}
		\text{div }(t^{1-2s}r^{n-1}\nabla V_s)=(n-1)r^{n-3}t^{1-2s}V_s \qquad&\text{ in } \mathbb{R}^{2}_{++}\\
		\left(-\partial_{rr}v_s-\frac{n-1}{r}\partial_rv_s+\frac{n-1}{r^2}v_s\right)r^{n-1}-d_s t^{1-2s}r^{n-1}\partial_t V_s|_{t\rightarrow 0}+r^{n-1}{v}_s=(p+1)	r^{n-1}\bar{u}_s^{p}v_s
		&\text{ on }  \mathbb{R}_+.
	\end{cases}
\end{equation}
It is known that $V_s=\partial_{r}\bar{U}_s$ doesn't change sign. Thus, we might as well assume that $V_s<0$ on $\mathbb{R}^2_{++}$ since the Maximum principle for elliptic equations.

We claim that  for every $i=1,\dots,n$
\begin{equation}\label{sdgsvd}
	G_1^i\equiv c^iV_s
\end{equation}
for some $c^i\in\mathbb{R}$. Indeed, for given $\Phi\in C^\infty_c(\mathbb{R}^2_{++}\cup \left\{t=0\right\})$, there exists $\Psi\in \mathcal{G}$ such that 
\[ \Phi={\Psi}{V_s}. \]
Thus, one has that  
\begin{equation}
	\begin{split}
		\int_{\mathbb{R}^2_{++}}t^{1-2s}r^{n-1}|\nabla\Phi|^2\, drdt&=\int_{\mathbb{R}^2_{++}}t^{1-2s}r^{n-1}|V_s\nabla\Psi|^2\, drdt+ \int_{\mathbb{R}^2_{++}}t^{1-2s}r^{n-1}\nabla V_s\cdot\nabla(V_s\Psi^2)\, drdt.
	\end{split}
\end{equation}
By testing~\eqref{vsdf} against $V_s\Psi^2$, we infer that 
\begin{equation}\label{dfve}
	\begin{split}
	d_s\int_{\mathbb{R}^2_{++}}t^{1-2s}r^{n-1}|V_s\nabla\Psi|^2\, drdt&=
		d_s\int_{\mathbb{R}^2_{++}}t^{1-2s}r^{n-1}\left(|\nabla\Phi|^2+(n-1)r^{-2}|\Phi|^2\right)\, drdt\\
		&\qquad+ \int_{\mathbb{R}_{+}}r^{n-1}\phi^2\, drdt-(p+1)\int_{\mathbb{R}_{+}}r^{n-1}u^p_s\phi^2\, drdt\\
		&\qquad+\int_{\mathbb{R}_{+}}r^{n-1}\left(\partial_r(v_s\psi^2)\partial_r v_s+(n-1)\phi^2r^{-2}\right)\, drdt.
	\end{split}
\end{equation}
Here $\phi:=$Tr$\Phi$, $v_s:=$Tr$V_s$ and  $\psi:=$Tr$\Psi$. In particular, we notice that  
\begin{equation}
	|\partial_{r}\phi|^2=|\partial_r(v_s\psi)|^2=\partial_r(v_s\psi^2)\partial_r v_s+v_s^2\partial_r^2\psi.
\end{equation}
From this, recalling~\eqref{vdsdfsd}, one has that
\begin{equation}
		d_s\int_{\mathbb{R}^2_{++}}t^{1-2s}r^{n-1}|V_s\nabla\Psi|^2\, drdt\leqslant I_1(\Phi,\Phi).
\end{equation}
Since $C^\infty_c(\mathbb{R}^2_{++}\cup \left\{t=0\right\})$ is dense in $\mathcal{G}$, owing to~\eqref{sdvsdv}, we infer that
\begin{equation}
	d_s\int_{\mathbb{R}^2_{++}}t^{1-2s}r^{n-1}|V_s\nabla(G_1^i V_s^{-1})|^2\, drdt\leqslant I_1(G_1^i,G_1^i)=0.
\end{equation}
This implies the desired result~\eqref{sdgsvd}.

As a result, we conclude that  $G_i^i(|x|,0)=c^i\partial_r\bar{U}_s(|x|,0)=c^i\partial_r\bar{u}_s(|x|)$ for all $x\in\mathbb{R}^n$. Recalling~\eqref{vdsvdsdv}, for every $w\in \text{Ker}\,L_s$, we obtain that
\begin{equation}
	w(x)=W(x,0)=G_0^1(|x|,0)+\sum_{i=1}^{n} G_1^i(|x|,0)\frac{x_i}{|x|}=G_0^1(|x|,0)+c^i\sum_{i=1}^{n} \partial_iu_s(x)
\end{equation}
From Lemmata~\ref{lemma nondegeneracy of radial sector} and~\ref{lemma nondegeneracy of radial sector s1}, we know that $G_0^1(|x|,0)=0$. This completes the proof of Theorem~\ref{th:nondegeneracy}.
\end{proof}

\section{Uniqueness}\label{sec:uniqueness}
Our goal this section is to prove Theorem~\ref{th:uniqueness}  by constructing a set of pseudo-minimizers. Based on the nondegeneracy result of Theorem~\ref{th:nondegeneracy}, we can now closely follow the approach provided in \cite{MR3207007}.

To begin with, we need to introduce the following fundamental estimate.   Let us  denote by $\bot_s$ the orthogonality relation in $H^{1}(\mathbb{R}^n)$.

\begin{Lemma}\label{lemma cnidnc}
Let $u_s\in\mathcal{M}_s$ and  $\mathcal{N}_s:=\left\{\varphi\in H^1_{\rm rad}(\mathbb{R}^n):\, \varphi\bot_s u_s \right\}$. 
We denote 
\begin{equation}\label{vdsbvsdv}
	\mathcal{J}_{u_s}^s(\varphi,v):=\frac{\left\langle L_s\varphi,v \right\rangle_{L^2}}{\left\langle \varphi,v \right\rangle_{s}}\quad \forall \varphi,v\in H^1_{\rm rad}(\mathbb{R}^n),
\end{equation}
and
\begin{equation}\label{csbdg}
	\Pi(s,u_s):=\inf\limits_{\mathcal{N}_s\backslash\left\{0\right\}} {\mathcal{J}_{u_s}^s(\varphi,\varphi)}.
\end{equation}
Then, there exist  $s_0, s_1\in(0,1)$ such  that  \begin{equation}
	\inf\limits_{s\in[0,s_0]\cup[s_1,1]}\inf\limits_{u_s\in \mathcal{M}_s} \Pi(s,u_s)>0.
\end{equation}
\end{Lemma}
\begin{proof}
From \eqref{ndsvnjsd}, it follows that $ \mathcal{J}_{u_s}^s(\varphi,\varphi)\geqslant 0$ for every $\varphi\bot_s u_s$. Thus, we assume by contradiction  that  
 there exists a sequence $s_k\rightarrow\sigma$ with  $\sigma\in\left\{0,1\right\}$  and $u_{s_k}\in \mathcal{M}_{s_k}$ such  that
\begin{equation}\label{vsdfds}
	 \Pi(s_k,u_{s_k})\rightarrow 0\qquad \text{as } k\rightarrow \infty.
\end{equation}
Moreover, for fixed $k\in\mathbb{N}$, by the definition of $\mathcal{J}_{u_{s_k}}^{s_k}(\varphi,\varphi)$, we are in a position of applying  the Ekeland variational principle, obtaining that there exists a minimizing sequence  $\varphi_{k,m}\in \mathcal{N}_{s_k}$ for $\Pi(s_k,u_{s_k}) $ such that $\|\varphi_{k,m}\|_{s_k}=1$ for all $m\in\mathbb{N}$ and $\nabla{\mathcal{J}_{u_{s_k}}^{s_k}(\varphi_{k,m},\varphi_{k,m})}\rightarrow 0$. 
%

Owing to the Riesz representation theorem, there exists $h_{k,m}\in\mathcal{N}_{s_k} $
 such that 
 \begin{equation}
 	\left\langle h_{k,m},v \right\rangle_{s_k}=\left\langle L_{s_k}\varphi_{k,m},v \right\rangle_{L^2}-\Pi(s_k,u_{s_k})\left\langle \varphi_{k,m},v \right\rangle_{s_k} \quad \forall v\in\mathcal{N}_{s_k}.
 \end{equation}
By taking into account the facts that $\left\langle L_{s_k}\varphi_{k,m},\varphi_{k,m} \right\rangle_{L^2}\rightarrow \Pi(s_k,u_k)$
 and  that
\begin{equation}
	\frac{1}{2}\left\langle \nabla{\mathcal{J}_{u_{s_k}}^{s_k}(\varphi_{k,m},\varphi_{k,m})},v\right\rangle_{s_k}=\left\langle L_{s_k}\varphi_{k,m},v \right\rangle_{L^2}-\left\langle L_{s_k}\varphi_{k,m},\varphi_{k,m} \right\rangle_{L^2}\left\langle \varphi_{k,m},v \right\rangle_{s_k}  
\end{equation}
for all $v\in  H^1_{\rm rad}(\mathbb{R}^n)$, we infer that $\|h_{k,m}\|_{s_k}\rightarrow 0$ as $k\rightarrow\infty.$

Thus, we can find a subsequence ${m_k}$ such that $\|h_{k,m_k}\|_{s_k}\rightarrow 0 $ as $k\rightarrow\infty$, and 
\begin{equation}\label{dvsdg}
	\left\langle h_{k,m_k},v \right\rangle_{s_k}=\left\langle L_{s_k}\varphi_{k,m_k},v \right\rangle_{L^2}-\Pi(s_k,u_{s_k})\left\langle \varphi_{k,m_k},v \right\rangle_{s_k} \quad \forall v\in\mathcal{N}_{s_k}.
\end{equation}

Assume that $w\in \mathcal{N}_{\sigma}\cap C^\infty_c(\mathbb{R}^n)$. Then, Lemma~\ref{lemma lambda1} implies that 
\[ \left\langle w,u_{s_k} \right\rangle_{s_k}\leqslant \|u_{s_k}\|_{s_k}\|w\|_{s_k}\leqslant C \|w\|_{\sigma}\]
for some positive constant $C$ independent of $s_k$ and $\sigma$.
Let us denote
 \[ v_{k}=w-\frac{\left\langle w,u_{s_k} \right\rangle_{s_k}}{\|u_{s_k}\|_{s_k}^2}u_{s_k}.  \]
From this, it follows that $v_k\in\mathcal{N}_{s_k}$ and $\|v_k\|_{s_k}\leqslant 4\|w\|_{\sigma}$.
Using $v_k$ as a test function in~\eqref{dvsdg}, on account of the facts that $\varphi_{k,m}\in \mathcal{N}_{s_k}$ and that $\|h_{k,m}\|_{s_k}\rightarrow 0$ as $k\rightarrow\infty$, we infer that 
\begin{equation}\label{fdvfsv}
	\left\langle L_{s_k}\varphi_{k,m_k},w \right\rangle_{L^2}-\Pi(s_k,u_{s_k})\left\langle \varphi_{k,m_k},w \right\rangle_{s_k}\rightarrow 0\quad \text{ as }k\rightarrow \infty.
\end{equation}

Since $\|\varphi_{k,m_k}\|_{s_k}=1$, we have $ \|\varphi_{k,m_k}\|_{\sigma}\leqslant 2 $ for all $k\in\mathbb{N}$. Thus, we may assume that up to a subsequence,  $\left\langle \varphi_{k,m_k},w \right\rangle_{\sigma} \rightarrow \left\langle \varphi,w \right\rangle_{\sigma}$ and by local Rellich compactness, we have that
\begin{equation}\label{fver}
	\varphi_{k,m_k}\rightarrow \varphi \qquad \text{strongly in } L^{p+2}_{\rm loc}(\mathbb{R}^n).
\end{equation}

 We now claim that $\varphi=0$. 
  The proof is divided into two steps, as follows:

  Step 1. We prove that \begin{equation}\label{cghgu}
 	\left\langle L_{\sigma}\varphi,w \right\rangle_{L^2}= 0\qquad \forall w\in \mathcal{N}_{\sigma}\cap C^\infty_c(\mathbb{R}^n).
 \end{equation}
Indeed, it is immediate to check that
\begin{equation}\label{fwefwe}
	\begin{split}
		\left|\left\langle \varphi_{k,m_k},w \right\rangle_{s_k}-\left\langle \varphi,w \right\rangle_{\sigma}\right|&\leqslant \left|\left\langle \varphi_{k,m_k}-\varphi,w \right\rangle_{\sigma}\right|+\left|\left\langle \varphi_{k,m_k},w \right\rangle_{s_k}-\left\langle \varphi_{k,m_k},w \right\rangle_{\sigma}\right|\\
		&\leqslant o(1)+|s_k-\sigma|\,  \|\varphi_{k,m_k}\|_{L^2(\mathbb{R}^n)}\|w\|_{3}.
	\end{split}
\end{equation}
Moreover, 
from Lemma~\ref{lemma convergence}, there exists $u_\sigma\in\mathcal{M}_\sigma$ such that 
\begin{equation}\label{vdsfs}
	\|u_{s_k}-u_\sigma\|_{2\sigma}\rightarrow 0\quad \text{as } k\rightarrow \infty.
\end{equation}
Thus, recalling~\eqref{dgsdbf},  we observe that, for every $R>0$ large enough,
\begin{equation}\label{dsvdsv}
	\begin{split}
	\left|\int_{{\mathbb{R}^n}}(u_{s_k}^{p}\varphi_{k,m_k}-u_\sigma^{p}\varphi)\,w\, dx\right|& \leqslant\int_{{\mathbb{R}^{n}}}|u_{s_k}^p-u^p_\sigma|\,|\varphi_{k,m_k}|\,|w|\, dx+\int_{{\mathbb{R}^{n}}}u_\sigma^p\,|\varphi_{k,m_k}-\varphi|\,|w|\, dx\\ &\leqslant o(1)+C\left(\|\varphi_{k,m_k}-\varphi\|_{L^{p+2}(B_R)}+R^{-np}\right)
	\end{split}
\end{equation}
for some constant $C>0 $ independent of $k$. From this and~\eqref{fver}, by combining~\eqref{fwefwe} and \eqref{vdsfs}, one has that 
\begin{equation}
	\begin{split}
		\left\langle L_{s_k}\varphi_{k,m_k},w \right\rangle_{L^2}& =\left\langle \varphi_{k,m_k},w \right\rangle_{s_k}-(p+1)\int_{{\mathbb{R}^n}}u_s^{p}\varphi_{k,m_k}\,w\, dx\\
		&\rightarrow \left\langle \varphi,w \right\rangle_{\sigma}-(p+1)\int_{{\mathbb{R}^n}}u_\sigma^{p}\varphi\,w\, dx=\left\langle L_{\sigma}\varphi,w \right\rangle_{L^2} \text{ as }k\rightarrow\infty.
	\end{split}
\end{equation} 
As a consequence of this, owing to~\eqref{vsdfds} and~\eqref{fdvfsv}, we obtain the desired result~\eqref{cghgu}.

Step 2. we show that 
\begin{equation}\label{vsdfvsd}
	\left\langle \varphi,u_\sigma \right\rangle_{\sigma}=0.
\end{equation}
Indeed, we observe that, 
if $\sigma=0$,  from~\eqref{bjk}, it follows that
\begin{equation}
	\begin{split}
		\left|\left\langle \varphi_{k,m_k},u_{s_k} \right\rangle_{s_k}-\left\langle \varphi,u_0 \right\rangle_{0}\right|&\leqslant \left|\left\langle \varphi_{k,m_k}-\varphi,u_0 \right\rangle_{0}\right|+\left|\left\langle \varphi_{k,m_k},u_0-u_{s_k} \right\rangle_{0}\right|\\
		&\quad +\left|\left\langle \varphi_{k,m_k}, u_{s_k} \right\rangle_{s_k}-\left\langle \varphi_{k,m_k}, u_{s_k} \right\rangle_{0}\right|\\
		&\leqslant o(1)+2\|u_0-u_{s_k}\|_{0} +C|s_k|\,\|u_{s_k}\|_{2},
	\end{split}
\end{equation}
where the constant $C>0$ is independent  of $k$.
By taking into account the fact that $ \left\langle \varphi_{k,m_k},u_{s_k} \right\rangle_{s_k}=0$ and Lemma~\ref{lemma 2 norm is bound},  one has $\left\langle \varphi,u_0 \right\rangle_{0}=0 $.

As for $\sigma=1$. Recall the fact that $\partial_j u_{s_k}\in \text{Ker}\,L_{s_k}$. We thus can infer from Lemma~\ref{lemma uniform decay of spectrum} that for every $s_k>9/10$, 
\begin{equation}
	\|\partial_j u_{s_k}\|_{2s_k}\leqslant C,
\end{equation}
where $C$ depends on $n,p,\|u_{s_k}\|_{L^\infty(\mathbb{R}^n)},\|\partial_j u_{s_k}\|_{L^2(\mathbb{R}^n)}$. Lemmata~\ref{lemma lambda1} and~\ref{lemma uniform bound} imply that $\|u_{s_k}\|_{2s_k+1}$ is uniformly bounded.  Also, one has that 
\begin{equation}
		\left|\left\langle \varphi_{k,m_k},u_{s_k} \right\rangle_{s_k}-\left\langle \varphi,u_1 \right\rangle_{1}\right|\leqslant o(1)+2\|u_1-u_{s_k}\|_{1} +C|1-s_k|\,\|u_{s_k}\|_{2s_k+1},
\end{equation}
where $C$ is independent of $s_k$. Hence, one deduces that $\left\langle \varphi,u_1 \right\rangle_{1}=0 $.  
This proves~\eqref{vsdfvsd}. 

By combining~\eqref{cghgu} with~\eqref{vsdfvsd}, we have that $\varphi\in \mathcal{N}_\sigma\cap \text{Ker}\,L_\sigma$.
 Thus, the nondegeneracy result of $L_\sigma$ implies $\varphi=0$.
\smallskip

Now we finish the proof of Lemma~\ref{lemma cnidnc}. Recalling~\eqref{fver}, one has that
\begin{equation}\label{vdsv }
	\varphi_{k,m_k}\rightarrow 0 \quad \text{in }L_{\rm loc}^{p+2}(\mathbb{R}^n).
\end{equation}
Moreover, by H\"{o}lder inequality and~\eqref{dgsdbf}, one has that, for $R>0$ large enough 
\begin{equation}
	\begin{split}
		\left\langle L_{s_k}\varphi_{k,m_k},\varphi_{k,m_k} \right\rangle_{L^2}&=1-(p+1)\int_{{\mathbb{R}^n}}u_{s_k}^p\varphi_{k,m_k}^2\\
		&\geqslant 1-(p+1)\left(\|u_s\|^p_{L^{p+2}(B_R)}\|\varphi_{k,m_k}\|^2_{L^{p+2}(B_R)}-CR^{-np}\|\varphi_{k,m_k}\|^2_{L^2(\mathbb{R}^n)}\right).
	\end{split}
\end{equation}
As a result, employing Lemma~\ref{lemma lambda1} and~\eqref{vdsv }, we see that
\begin{equation}\label{afaddgv}
	\left\langle L_{s_k}\varphi_{k,m_k},\varphi_{k,m_k} \right\rangle_{L^2}= 1-o(1).
\end{equation}
Additionally, from~\eqref{vsdfds} and~\eqref{dvsdg}, one derives that 
\[ \left\langle L_{s_k}\varphi_{k,m_k},\varphi_{k,m_k} \right\rangle_{L^2}\rightarrow 0\qquad \text{as } k\rightarrow \infty. \]
This  contradicts~\eqref{afaddgv}.  
The proof of Lemma~\ref{lemma cnidnc} is thereby complete.
\end{proof}

%
%

Let $s\in[0,1]$, and  denote a functional 
\begin{equation}\label{dfdsvnk}
	\mathcal{I}_s(u):=\frac{1}{2}\|u\|_{s}^2-\frac{1}{p+2}\int_{{\mathbb{R}^n}}|u|^{p+2}\, dx \quad \forall u\in H^1(\mathbb{R}^n).
\end{equation}
Also, given $\sigma\in[0,1]$,  we define a mapping $\Phi_s^\sigma:H^1_{\rm rad}(\mathbb{R}^n)\rightarrow H^1_{\rm rad}(\mathbb{R}^n) $ by
\begin{equation}\label{dsgrvx}
	\Phi^\sigma_s(w)=\nabla \mathcal{I}_s(u_\sigma+w),
\end{equation}
where  $s\leqslant \sigma$ and  $u_s\in\mathcal{M}_s$.
As customary,  \eqref{dsgrvx} is equivalent to  
\begin{equation}
	\left\langle \Phi^\sigma_s(w),v \right\rangle_{s}:=\left\langle \nabla \mathcal{I}_s(u_\sigma+w),v \right\rangle_{L^2}\qquad \forall v\in H^1_{\rm rad}(\mathbb{R}^n).
\end{equation}

\subsection{$s$ close to $1$}\label{sec nvfnvk}

In view of the  uniqueness result for the local case, we first prove that if $s$ is sufficiently close to $1$, the equation possesses a unique $u_s\in\mathcal{M}_s$, as given by the following result.

\begin{Theorem}\label{th:uniqueness s close to 1}
	There exists $s_1\in(0,1)$ such that for every $s\in(s_1,1)$, the ground state $u_s$ is unique up to translation.
\end{Theorem}

To begin with, we need to introduce the following Lemmata~\ref{lemma busdhiv} and~\ref{lemma mlfbdfzb}.

\begin{Lemma}\label{lemma busdhiv}
	For every $f\in H^1_{\rm rad}(\mathbb{R}^n)$, there exists a unique $ w_s\in H^1_{\rm rad}(\mathbb{R}^n)$ such that 
	\begin{equation}\label{vkjn}
		 \left\langle \nabla\Phi^1_s(0)[w_s],w \right\rangle_{s}= \left\langle f,w \right\rangle_{s}\qquad \forall w\in H^1_{\rm rad}(\mathbb{R}^n).
	\end{equation}
In particular, there exist $C>0$ and $s_1\in(0,1)$ such that 
\begin{equation}\label{xfhv}
	\left|\left(\nabla\Phi^1_s(0)\right)^{-1}\right|\leqslant C \qquad \forall s\in(s_1,1).
\end{equation}
\end{Lemma}
\begin{proof}
To this end, we first notice that 
\begin{equation}
	\left\langle \nabla\Phi^1_s(0)[v],w \right\rangle_{s}= \left\langle L_1 v,w \right\rangle_{L^2}\qquad \forall v,w\in H^1_{\rm rad}(\mathbb{R}^n).
\end{equation}
Hence, it suffices to find a solution $w_s$ to the equation
\begin{equation}\label{dfvdf}
	\left\langle L_1 w_s,w \right\rangle_{L^2}=\left\langle f,w \right\rangle_{s} \qquad \forall w\in H^1_{\rm rad}(\mathbb{R}^n).
\end{equation}
For this, by Lemma~\ref{lemma convergence}, we know that $\|u_s-u_1\|_{2}\rightarrow 0$ as $s\nearrow 1$. This implies that $u_s\rightarrow u_1$ in $L^{p+2}(\mathbb{R}^n)$ and thus we have  that, for every $w\in H^1_{\rm rad}(\mathbb{R}^n), $
\begin{equation}
	\begin{split}
	\left|\left\langle L_1 w,w \right\rangle_{L^2}-\left\langle L_s w,w \right\rangle_{L^2}\right|&=(p+1)\left|\int_{{\mathbb{R}^n}}\left(u_1^p-u_s^p\right)w^2\, dx\right|\\
	&\leqslant (p+1)\|u_1^p-u_s^p\|^p_{L^{\frac{p+2}{p}}(\mathbb{R}^n)}\|w\|^2_{L^{p+2}(\mathbb{R}^n)}\\
	&\leqslant (p+1)\|u_1^{p+2}-u_s^{p+2}\|^p_{L^{1}(\mathbb{R}^n)}\|w\|^2_{L^{p+2}(\mathbb{R}^n)}.
	\end{split}
\end{equation}
It follows that 
\begin{equation}\label{vjhhkj}
		\left|\left\langle L_1 w,w \right\rangle_{L^2}\right|\geqslant\left|\left\langle L_s w,w \right\rangle_{L^2}\right|-o(1)\|w\|^2_{s}.
\end{equation}
Moreover, we observe that \begin{equation}\label{dfis}
	\left|\left\langle L_s u_s,u_s \right\rangle_{L^2}\right|=p\|u_s\|^2_{s}.
\end{equation}
Additionally, in the light of Lemma~\ref{lemma cnidnc}, one has that 
\begin{equation}\label{vdsnkvji}
	\left|\left\langle L_s w,w \right\rangle_{L^2}\right|\geqslant C\|w\|^2_{s},\qquad \forall w\bot_s u_s
\end{equation}
for some positive constant $C$.

By combining~\eqref{vjhhkj}, ~\eqref{dfis} and~\eqref{vdsnkvji}, one deduces that
\begin{equation}\label{bfdg}
	\left|\left\langle L_1 v,v \right\rangle_{L^2}\right|\geqslant C\|v\|^2_{s},\qquad \forall v\in H^1_{\rm rad}(\mathbb{R}^n).
\end{equation}
As a consequence, we are in a position to apply the Lax-Milgram Theorem, and so there exists a unique $w_s\in H^1_{\rm rad}(\mathbb{R}^n)$ such that 
 \begin{equation}
 	\left\langle L_1 w_s,w \right\rangle_{L^2}=\left\langle f,w \right\rangle_{s} \qquad \forall w\in H^1_{\rm rad}(\mathbb{R}^n).
 \end{equation}
Also, by~\eqref{bfdg}
\begin{equation}
	\|f\|_{s}\geqslant \frac{\left|\left\langle L_1 w_s,w_s \right\rangle_{L^2}\right|}{\|w_s\|_s}\geqslant C\|w_s\|_s.
\end{equation}
This entails the desired results ~\eqref{vkjn} and~\eqref{xfhv}.
\end{proof}

Given $\sigma\in[0,1]$, for $r>0$ and $s\in(0,1)$,   we set 
\begin{equation}\label{vhjjbbj}
	\mathcal{A}^\sigma_{r,s}:=\left\{w\in H^1_{\rm rad}(\mathbb{R}^n):\, \|w\|_s\leqslant r|\sigma-s| \right\}.
\end{equation}

\begin{Lemma}\label{lemma mlfbdfzb}
 There exist $s_1\in(0,1)$ and $r_0>0$ such that for any $s\in(s_1,1)$, one can find a unique function $w^s\in \mathcal{A}^1_{r_1,s_1}$ such that 
 \begin{equation}
 	\Phi^1_s(w^s)=0.
 \end{equation}
Here $\mathcal{A}^1_{r_1,s_1}$ is defined in \eqref{vhjjbbj}.
\end{Lemma}
\begin{proof}
On account of Lemma~\ref{lemma busdhiv}, we can entitle to transform 
the equation $	\Phi^1_s(w)=0$  to the following fixed point equation
\begin{equation}\label{vfdger}
	w=-\left(\nabla\Phi^1_s(0)\right)^{-1}\left(\Phi^1_s(0)+\Psi^1_s(w)\right)
\end{equation}
where $\Psi^1_s(w):=-\Phi^1_s(0)+\Phi^1_s(w)-\nabla\Phi^1_s(0)[w]$. Since $u_1$ is radial, from the definition of $\Phi_s^1$, it follows that $ \Psi_s^1$ is radial. 

Furthermore, we observe that, for every $v\in H^1_{\rm rad}(\mathbb{R}^n)$
\begin{equation}\label{fsvfd}
	\begin{split}
	\left\langle \Psi^1_s(w),v \right\rangle_{s}&=-\left\langle \Phi^1_s(0),v \right\rangle_{s}+\left\langle \Phi^1_s(w),v \right\rangle_{s}-\left\langle \nabla\Phi^1_s(0)[w],v \right\rangle_{s}\\
	&=\int_{{\mathbb{R}^n}}u_1^{p+1}v\, dx-\int_{{\mathbb{R}^n}}(u_1+w)^{p+1}v\,dx+(p+1)\int_{{\mathbb{R}^n}}u_1^pwv\, dx.
	\end{split}
\end{equation}
In the light of the  Mean value Theorem, we see that 
\begin{equation}
\left|(u_1+w)^{p+1}-u_1^{p+1}-(p+1)u_1^pw\right|\leqslant\begin{cases}
		C_p |w|^{p+1}  & 0<p<1\\
		C_p \left(|w|^2\|u_1\|^{p-1}_{L^\infty(\mathbb{R}^n)}+|w|^{p+1}\right) &p\geqslant 1.
	\end{cases}
\end{equation}
From this and~\eqref{fsvfd}, it follows that
\begin{equation}\label{vsdgsdv}
	\left|\left\langle \Psi^1_s(w),v \right\rangle_{s}\right|\leqslant	\begin{cases}
	 C_p\|w\|^{p+1}_{L^{p+2}(\mathbb{R}^n)}\|v\|_{L^{p+2}(\mathbb{R}^n)} & 0<p<1\\
	 C_{p,u_1}\left(\|w\|^2_{L^{\frac{2q}{q-1}}(\mathbb{R}^n)}\|v\|_{L^{q}(\mathbb{R}^n)}+\|w\|^{p+1}_{L^{p+2}(\mathbb{R}^n)}\|v\|_{L^{p+2}(\mathbb{R}^n)} \right) & p\geqslant 1,
	\end{cases}
\end{equation}
 where $q$ is either equal to the critical exponent $2^*$ if $n > 2$ or any real number in $(p+2,+\infty)$
	if $n =1,2$. 
Let us denote $\gamma:=\min\left\{2,p+1\right\}$.  By Sobolev inequality, we then have that for every $\|w\|_s<1$,
\begin{equation}\label{vsdg}
	\|\Psi^1_s(w)\|_s\leqslant C_1\|w\|^\gamma_s
\end{equation}
for some positive constant $C_1$ independent of $s$.

In addition, we have that 
\begin{equation}\label{vjb k}
	\begin{split}
	&\left|(u_1+w_1)^{p+1}-(u_1+w_2)^{p+1}-(p+1)u_1^p(w_1-w_2)\right|\\
\leqslant&\begin{cases}
		C_p \|u_1\|_{L^\infty(\mathbb{R}^n)}\left(|w_1|^{p}+|w_2|^p\right)|w_1-w_2|  \qquad    \qquad    \qquad   \qquad   \qquad \qquad 0<p<1\\
		C_p \left(\|u_1\|_{L^\infty(\mathbb{R}^n)}\left(|w_1|^{p}+|w_2|^p\right)+\|u_1\|^p_{L^\infty(\mathbb{R}^n)}\left(|w_1|+|w_2|\right)\right)|w_1-w_2|\quad  p\geqslant 1.
	\end{cases}
\end{split}
\end{equation}
It follows that 
\begin{equation}\label{vfdvdf}
	\begin{split}
	\|\Psi^1_s(w_1)-\Psi^1_s(w_2)\|_s
	\leqslant\begin{cases}
	 C_{p,u_1}\left(\|w_1\|^{p}_{s}+\|w_2\|^{p}_{s}\right)\|w_1-w_2\|_{s} \qquad \qquad  &0<p<1\\	 
	 C_{p,u_1}\left(\|w_1\|^{p}_{s}+\|w_2\|^{p}_{s}+\|w_1\|_{s}+\|w_2\|_{s}\right)\|w_1-w_2\|_{s}   &p\geqslant1.
 \end{cases}
\end{split}
\end{equation}
This implies that for every $\|w_1\|_s<1$ and $\|w_2\|_s<1$,
\begin{equation}\label{vsfvs}
	\|\Psi^1_s(w_1)-\Psi^1_s(w_2)\|_s\leqslant C_2\left(\|w_1\|^{\gamma-1}_{s}+\|w_2\|^{\gamma-1}_{s}\right)\|w_1-w_2\|_{s},
\end{equation}
where the positive constant $C_2$ is independent of $s$.

Now we claim that there exists a positive  constant $C_3$ independent of $s$ such that
\begin{equation}\label{nvksdvj}
	\|\Phi_s^1(0)\|_s\leqslant C_3(1-s).
\end{equation}

Indeed, by taking into account the fact that $ \left\langle \nabla \mathcal{I}_1(u_1),v \right\rangle_{L^2}=0$ for every $v\in H^1_{\rm rad}(\mathbb{R}^n)$, we see that for every $s>s_1>\frac{4}{5}$ 
\begin{equation}
\begin{split}
\left|\left\langle \Phi^1_s(0),v \right\rangle_{s}\right|&=\left|\left\langle \nabla \mathcal{I}_s(u_1)-\nabla\mathcal{I}_1(u_1),v \right\rangle_{L^2}\right|\leqslant \int_{{\mathbb{R}^n}}\left||\xi|^2-|\xi|^{2s}\right||\hat{u}_1|\,|\hat{v}|\, d\xi\\
&\leqslant C_{n}(1+\frac{1}{\delta})|1-s| \int_{{\mathbb{R}^n}}\left(1+|\xi|^{2+\delta+2-2s}\right)|\hat{u}_1|\,|\hat{v}|\, d\xi\leqslant C_{n}|1-s|\, \|u_1\|_{2}\, \|v\|_{s},
\end{split}
\end{equation}
where we set $\delta=2s_1-1$. This gives the claim~\eqref{nvksdvj}.

Finally, we shall prove that there exist  $s_1\in(0,1)$ and $r_1\in(0,1)$  such that for every $s\in(s_1,1)$, one can find a unique function $w^s\in \mathcal{A}^1_{r_1,s_1}$  solving the fixed equation~\eqref{vfdger}.

Indeed, from~\eqref{xfhv},~\eqref{vsdg} and~\eqref{nvksdvj}, it follows that,  for every $w\in \mathcal{A}^1_{r,s}$ 
\begin{equation}
	\|-\left(\nabla\Phi^1_s(0)\right)^{-1}\left(\Phi^1_s(0)+\Psi^1_s(w)\right)\|_s\leqslant C\left(C_3(1-s)+C_1r^\gamma(1-s)^\gamma\right).
\end{equation}  
Moreover, since $\gamma=\min\left\{2,p+1\right\}>1$, thus there exists $r_1>0$ sufficiently large and $s_1\in(0,1)$ depending on $r_1$ such that for any $s\in(s_1,1)$, one has that 
\begin{equation}
		r_1(1-s_1)> C\left(C_3(1-s)+C_1r_1^\gamma(1-s_1)^\gamma\right).
\end{equation}
As a result, for every $s\in(s_1,1)$, the mapping 
\begin{equation}
T:	w\longmapsto -\left(\nabla\Phi^1_s(0)\right)^{-1}\left(\Phi^1_s(0)+\Psi^1_s(w)\right)
\end{equation} 
maps $\mathcal{A}^1_{r_1,s_1}$ into itself. 

Moreover, by combining~\eqref{xfhv} with \eqref{vsfvs}, for every $w_1,w_2\in \mathcal{A}^1_{r_1,s_1}$,  one has that
\begin{equation}
	\|Tw_1-Tw_2\|_s\leqslant C\|\Psi^1_s(w_1)-\Psi^1_s(w_2)\|_s\leqslant 2CC_2\left(r_1(1-s_1)\right)^{\gamma-1}\|w_1-w_2\|_{s}.
\end{equation}
As a consequence of this, increasing $s_1$ if necessary,  the map $T$ is a contraction on $\mathcal{A}^1_{r_1,s_1}$. By the Banach fixed point theorem, for any 
$s\in(s_1,1)$, there exists a unique function $w^s\in \mathcal{A}^1_{r_1,s_1}$ solving the equation~\eqref{vfdger}. This ends the proof of Lemma~\ref{lemma mlfbdfzb}.
\end{proof}

With the preliminary work done so far, we can now complete the proof of Theorem~\ref{th:uniqueness s close to 1}. 
\begin{proof}[Proof of Theorem~\ref{th:uniqueness s close to 1}]
	Let $u_s^1\in\mathcal{M}_s$ and $u_s^2\in\mathcal{M}_s$. 
	
	We now claim that $u_s^1=u_s^2$ provided $s$ is close to $1$.
	
	Indeed, by combining Lemma~\ref{lemma convergence} and the uniqueness result for local case, we know that $u_s^i=u_1+w_s^i$ with $\|w_s^i\|_2\rightarrow 0$ as $s\nearrow 1$, for $i=1,2$. Since $w_s^i$ is symmetric with respect to the origin for $i=1,2$, on account of~\eqref{dsgrvx}, one has that $\Phi^1_s(w_s^i)=0$. Thus, by Lemma~\ref{lemma mlfbdfzb}, we conclude  that $w_s^1=w_s^2$
	 when $s$ is close to $1$.  
\end{proof}

\subsection{$s$ close to $0$}
We now turn to the unique result when $s$ close to $0$ by following the general strategy derived in Section~\ref{sec nvfnvk}.
Using a similar argument to the proof of Lemma~\ref{lemma busdhiv}, we have

\begin{Lemma}\label{lemma fvjhjjb}
	For every $f\in H^1_{\rm rad}(\mathbb{R}^n)$, there exists a unique $ w_s\in H^1_{\rm rad}(\mathbb{R}^n)$ such that 
	\begin{equation}\label{vsfsd}
		\left\langle \nabla\Phi^{0}_s(0)[w_s],w \right\rangle_{s}= \left\langle f,w \right\rangle_{s}\qquad \forall w\in H^1_{\rm rad}(\mathbb{R}^n).
	\end{equation}
	In particular, there exist $C>0$ and $s_0\in(0,1)$ such that 
	\begin{equation}\label{zxvdvd}
		\|\left(\nabla\Phi^{0}_s(0)\right)^{-1}\|\leqslant C \qquad \forall s\in(0,s_0).
	\end{equation}
\end{Lemma}

\begin{Lemma}\label{lemma bdfbfd}
	There exists $s_0\in(0,1)$, $r_0>0$ such that for any $s\in(0,s_0)$, there exists a unique function $w^s\in \mathcal{A}^{0}_{r_0,s_0}$ such that 
	\begin{equation}
		\Phi^{0}_s(w^s)=0.
	\end{equation}
	Here $\mathcal{A}^{0}_{r_0,s_0}$ is defined in \eqref{vhjjbbj}.
\end{Lemma}
\begin{proof}
	We observe that Lemma~\ref{lemma fvjhjjb} allows us to transform 
	the equation $	\Phi^{0}_s(w)=0$  to the following fixed point equation
	\begin{equation}\label{bdfhtd}
		w=-\left(\nabla\Phi^{0}_s(0)\right)^{-1}\left(\Phi^{0}_s(0)+\Psi^{0}_s(w)\right)
	\end{equation}
	where $\Psi^{0}_s(w):=-\Phi^{0}_s(0)+\Phi^{0}_s(w)-\nabla\Phi^{0}_s(0)[w]$. Since $u_{0}\in\mathcal{M}_{0}$ is radial, by the definition of $\Phi_s^{0}$, it follows that $ \Psi_s^{0}$ is radial.

%

Furthermore, 	by repeating the procedure in Lemma~\ref{lemma mlfbdfzb}, we infer that 
for every $v\in H^1_{\rm rad}(\mathbb{R}^n)$ and $\|w_1\|_s<1$ and $\|w_2\|_s<1$,
	\begin{equation}\label{vsdvs}
			\|\Psi^{0}_s(w_1)\|\leqslant C_1\|w_1\|^\gamma_s,
		\end{equation}
and
\begin{equation}\label{ vdvdsv}
	\|\Psi^{0}_s(w_1)-\Psi^{0}_s(w_2)\|\leqslant C_2\left(\|w_1\|^{\gamma-1}_{s}+\|w_2\|^{\gamma-1}_{s}\right)\|w_1-w_2\|_{s},
\end{equation}
where the positive constants $C_1$ and $C_2$ is independent of $s$, and $\gamma=\min\left\{p+1,2\right\}$.

%
%
%

Now we claim that there exists a positive  constant $C_3$ independent of $s$ such that
\begin{equation}\label{nvksdvjj}
	\|\Phi_s^0(0)\|_s\leqslant C_3s.
\end{equation}

Indeed, using the fact that $ \left\langle \nabla \mathcal{I}_0(u_0),v \right\rangle_{L^2}=0$ for every $v\in H^1_{\rm rad}(\mathbb{R}^n)$, owing to~\eqref{bjk}, one derives that, for every $s<1/2$
\begin{equation}
		\left|\left\langle \Phi^0_s(0),v \right\rangle_{s}\right|=\left|\left\langle \nabla \mathcal{I}_s(u_0)-\nabla\mathcal{I}_0(u_0),v \right\rangle_{L^2}\right|\leqslant \int_{{\mathbb{R}^n}}\left|1-|\xi|^{2s}\right|\,|\hat{u}_0|\,|\hat{v}|\, d\xi\leqslant C\, s\|u_0\|_{2}\|v\|_{L^2(\mathbb{R}^n)}
\end{equation}
 where the constant $C>0$ is independent of $s$. This gives the claim~\eqref{nvksdvjj}.

	

	Finally, we shall prove that there exist  $s_0, r_0\in(0,1)$  such that for every $s\in(0,s_0)$, one can find a unique function $w^s\in \mathcal{A}^{0}_{r_0,s_0}$  solving the fixed equation~\eqref{bdfhtd}.
	
	Indeed, from~\eqref{zxvdvd},~\eqref{vsdvs} and~\eqref{nvksdvjj}, it follows that,  for every $w\in \mathcal{A}^{0}_{r,s}$ 
	\begin{equation}
		\|-\left(\nabla\Phi^{0}_s(0)\right)^{-1}\left(\Phi^{0}_s(0)+\Psi^{0}_s(w)\right)\|\leqslant C\left(C_3s+C_1r^\gamma\,s^\gamma\right).
	\end{equation}  
	Moreover, since $\gamma=\min\left\{2,p+1\right\}>1$, thus there exist $r_0>0$ sufficiently large and $s_0\in(0,1)$ depending on $r_0$ such that for any $s\in(0,s_0)$, one has that 
	\begin{equation}
		r_0\,s_0> C\left(C_3s+C_1r_0^\gamma\,s_0^\gamma\right).
	\end{equation}
	As a result, for every $s\in(0,s_0)$, the mapping 
	\begin{equation}
		T_{0}:	w\longmapsto -\left(\nabla\Phi^{0}_s(0)\right)^{-1}\left(\Phi^{0}_s(0)+\Psi^{0}_s(w)\right)
	\end{equation} 
	maps $\mathcal{A}^{0}_{r_0,s_0}$ into itself. 
	
	Moreover, by combining~\eqref{zxvdvd} with \eqref{ vdvdsv}, for every $w_1,w_2\in \mathcal{A}^{0}_{r_0,s_0}$,  one has that
	\begin{equation}\label{bvjdsvnjsd}
		\|T_{0}w_1-T_{0}w_2\|_s\leqslant C\|\Psi^{0}_s(w_1)-\Psi^{0}_s(w_2)\|_s\leqslant 2CC_2\left(r_0\,s_0\right)^{\gamma-1}\|w_1-w_2\|_{s}.
	\end{equation}
	As a consequence of this, increasing $s_0$ if necessary,  the map $T_{0}$ is a contraction on $\mathcal{A}^{0}_{r_0,s_0}$. By the Banach fixed point theorem, for any 
	$s\in(0,s_0)$, there exists a unique function $w^s\in \mathcal{A}^{0}_{r_0,s_0}$ solving the equation~\eqref{bdfhtd}. This ends the proof of Lemma~\ref{lemma bdfbfd}.
\end{proof}

 We  now complete the proof of Theorem~\ref{th:uniqueness}.
 
\begin{proof}[Proof of Theorem~\ref{th:uniqueness}]
	Let $u_s^1\in\mathcal{M}_s$ and $u_s^2\in\mathcal{M}_s$. 
	
	We now claim that $u_s^1=u_s^2$ provided $s$ is close to $0$.
	
	Indeed, by combining Lemma~\ref{lemma convergence} and the uniqueness result for $s=0$, we know that $u_s^i=u_{0}+w_s^i$ with $\|w_s^i\|_{2}\rightarrow 0$ as $s\searrow 0$, for $i=1,2$. Since $w_s^i$ is symmetric with respect to the origin for $i=1,2$, on account of~\eqref{dsgrvx}, one has that $\Phi^{0}_s(w_s^i)=0$. Thus, by Lemma~\ref{lemma bdfbfd}, we conclude  that $w_s^1=w_s^2$
	when $s$ is close to $0$.
	
	As a result, from this and Theorem~\ref{th:uniqueness s close to 1}, we can conclude that there exist  $s_0,s_1$ such that for every $s\in(0,s_0)\cup(s_1,1)$, the solution of problem~\eqref{main equation} $u_s\in\mathcal{M}_s$ is unique, up to a translation, we can obtain Theorem~\ref{th:uniqueness}, as desired.
	 \end{proof}

\appendix

\section{Existence and Properties of Ground states}\label{sec:Existence and Properties of Ground states }
	In this part, we provide some details of the proof of Theorem~\ref{th: existence and properity} inferred from \cite{DSVZ24}.

	 \smallskip
	 
	{Step 1.  Existence (P.L.Lions, 1984).} We first denote $I_s(u):=\|u\|_{s}^2$. Let $ \left\{u_k\right\}\subset H^1(\mathbb{R}^n)$ with $\|u_k\|_{L^{p+2}(\mathbb{R}^n)}=1$ be a minimizing sequence for $I_s(u)$, that is,
	\[ \lim\limits_{k\rightarrow\infty}I_s(u_k)=\lambda_s.\] 
 Since $ I_s(|u|)\leqslant I_s(u)$ holds without loss of generality, we can suppose that $u_k$ is nonnegative. It is clear that  the infimum $\lambda_s$ is bounded, thus one has that $u_k$ is uniformly bounded in $H^1(\mathbb{R}^n)$. Hence, there exists $u\in H^1(\mathbb{R}^n)$ such that $u_k\rightharpoonup u$ weakly in $H^1(\mathbb{R}^n)$ and pointwise a.e. in $\mathbb{R}^n$. Fatou Lemma and  \cite[Lemma~2.18]{MR1181725} imply that 
	\begin{equation*}
	\delta:=	\liminf_{k\to+\infty}\sup\limits_{y\in\mathbb{R}^n}\int_{B_{1}(y)}|u_k(x)|^2\, dx>0.
	\end{equation*}
Going if necessary to a subsequence, we may assume the existence of
$y_k\in\mathbb{R}^n$ such that
\begin{equation}\label{vsdfaf}
	\int_{B_{1}(y_k)}|u_k(x)|^2\, dx>{\delta}/{2}.
\end{equation}
Let us define $v_k(x):=u_k(x+y_k)$. Hence  $\|v_k\|_{L^{p+2}(\mathbb{R}^n)}=1$, $\|v_k\|_{s}^2\rightarrow \lambda_s$, and 
\begin{equation}\label{sdvdfd}
	\int_{B_{1}(0)}|v_k(x)|^2\, dx>{\delta}/{2}.
\end{equation}
Since $v_k$ is bounded in $H^1(\mathbb{R}^n)$,   there exists $v\in H^1(\mathbb{R}^n)$ such that $v_k\rightharpoonup v$ weakly in $H^1(\mathbb{R}^n)$ and pointwise a.e. in $\mathbb{R}^n$.
By Lemma~1.32 in \cite{Willembook}, one has that 
\begin{equation}\label{dsvds}
	1=\|v\|^{p+2}_{L^{p+2}(\mathbb{R}^n)}+\lim\limits_{k\rightarrow\infty}\|v-v_k\|^{p+2}_{L^{p+2}(\mathbb{R}^n)}.
\end{equation}
Hence, we deduce that
\begin{equation}\label{fdvdsv}
	\begin{split}
\lambda_s=\lim\limits_{k\rightarrow\infty}\|v_k\|_{s}^2=&\|v\|_{s}^2+\lim\limits_{k\rightarrow\infty}\|v-v_k\|_{s}^2\\
\geqslant&\lambda_s\left(\|v\|^2_{L^{p+2}(\mathbb{R}^n)}+\left(1-\|v\|^{p+2}_{L^{p+2}(\mathbb{R}^n)}\right)^{2/(p+2)}\right)\geqslant \lambda_s.	
\end{split}
\end{equation}
The last inequality holds since $ a^{\frac{2}{p+2}}+(1-a)^{\frac{2}{p+2}}\geqslant 1$ for every $a\in[0,1].$   
By combining~\eqref{sdvdfd} and Fatou Lemma, one has that $v\neq 0$. Thus, we obtain that $\|v\|^{p+2}_{L^{p+2}(\mathbb{R}^n)}=1$ and $\|v\|_{s}^2=\lambda_s=\lim\limits_{k\rightarrow\infty}\|v_k\|_{s}^2 $. This implies that $v$ is a
minimizer for $\lambda_s$.

Moreover, we note that the minimizer $v$ satisfies 
	\[ \partial\bigg|_{\epsilon=0}J_s(v+\epsilon\varphi)=0 \quad \forall \varphi\in C_0^\infty(\mathbb{R}^n). \]
	A calculation indicates that the function $v$ solves the 
	\begin{equation}
			- \Delta v +  (-\Delta)^s v+v = \lambda_sv^{p+1} \quad \hbox{in $\mathbb{R}^n$.}
		\end{equation}
	Moreover, we define by $	u_s:=\lambda_s^{\frac{1}{p}}v$
	belonging to $ \mathcal{M}_s$.

	\bigskip
	
	{Step 2. Regularity and decay of solutions.}

	 To begin with, we prove the following $H^2$-regularity Lemma, which in turn will serve as the basic step to prove the  $C^2$-regularity of solutions to equation~\eqref{main equation} when $n=1$. 
\begin{Lemma}\label{lemma H^2}
Let $n\geqslant 1$ and  $u\in H^1(\mathbb{R}^n)\cap L^\infty(\mathbb{R}^n)$ be a weak solution of equation~\eqref{main equation}. Then $u\in H^2(\mathbb{R}^n)$. Moreover,  there exists a constant $C>0$  such that 
\begin{equation}\label{H^2 estimate}
	\|u\|_{H^2(\mathbb{R}^n)}\leqslant C\left(\|u\|_{H^1(\mathbb{R}^n)},\|u\|_{L^\infty(\mathbb{R}^n)},n,s,p\right).
\end{equation}
\end{Lemma}
\begin{proof}
	Let $\varphi\in C^\infty_0(\mathbb{R}^n,\mathbb{R})$ be a cutoff function satisfying 
	\begin{itemize}
	\item[(i)]	$\varphi\equiv1$ on $B_R:=\left\{x\in\mathbb{R}^n:\, |x|\leqslant R\right\}$ with $R\gg1$ and supp$(\varphi)\subset B_{2R}$;
	\item[(ii)]  $0\leqslant \varphi\leqslant 1$  and $|\nabla \varphi|\leqslant \frac{2}{R}<1$ on $\mathbb{R}^n$.
	\end{itemize}
	For every fixed $k\in\left\{1,\dots,n\right\}$ and every $0<|h|<R$, we set 
	\[ v:=-D_k^{-h}(\varphi^2 D_k^{h}u), \quad\text{ where }  D_k^{h}w(x):=\frac{w(x+he_k)-w(x)}{h}.\]
 We notice that, since $u\in H^1(\mathbb{R}^n)$, then $D_k^{h}u\in H^1(\mathbb{R}^n)$. In view of (i), we see that $v\in H^1(\mathbb{R}^n)$ and supp$(v)\subset B_{3R}.$ 
 
 Therefore, we are in a position of using $v$ as a test function in equation~\eqref{main equation}, obtaining that 
 \begin{equation}
 	\begin{split}\label{dvsvss}
 	\int_{{\mathbb{R}^n}}&\varphi^2 |D_k^h(\nabla u)|^2\, dx +\int_{{\mathbb{R}^n}} 2\varphi\, D_k^hu\, \nabla\varphi\, D_k^h(\nabla u)\, dx\\
 	&+\frac{c_{n,s}}{2} \int_{{\mathbb{R}^{2n}}}\frac{(D_k^hu(x)-D_k^hu(y))(\varphi^2(x)D_k^hu(x)-\varphi^2(y)D_k^hu(y))}{|x-y|^{n+2s}}\, dxdy\\
 	&=\int_{{\mathbb{R}^n}} (-u+u^{p+1}) v\, dx.
 	\end{split}
 \end{equation}
	Now, by   the classical Young inequality, we obtain the following estimate
	
	\begin{equation}
		\begin{split}\label{bjckx}
		\left|\int_{{\mathbb{R}^n}} 2\varphi\, D_k^hu\, \nabla\varphi\, D_k^h(\nabla u)\, dx\right|&\leqslant 2\int_{{\mathbb{R}^n}} 2\varphi\, |D_k^hu|\, |\nabla\varphi|\, |D_k^h(\nabla u)|\, dx\\
		&\leqslant \frac{1}{2} \int_{{\mathbb{R}^n}} \varphi^2\, |D_k^h(\nabla u)|\, dx+2 (\sup\limits_{\mathbb{R}^n}|\nabla\varphi|)\int_{{\mathbb{R}^n}}|D_k^hu|\, dx\\
		&\leqslant  \frac{1}{2} \int_{{\mathbb{R}^n}} \varphi^2\, |D_k^h(\nabla u)|\, dx+C_n\int_{{\mathbb{R}^n}}|\nabla u|^2\, dx,
		\end{split}
	\end{equation}
	where in the last inequality, we used \cite[Chapter~5.8.2, Theorem~3]{MR2597943}.  By combining~\eqref{dvsvss} with~\eqref{bjckx}, we see that 
	\begin{equation}
		\begin{split}\label{vsdvsd}
			\frac{1}{2}\int_{{\mathbb{R}^n}}&\varphi^2 |D_k^h(\nabla u)|^2\, dx -C_n\int_{{\mathbb{R}^n}}|\nabla u|^2\, dx+\frac{c_{n,s}}{2} \int_{{\mathbb{R}^{2n}}}\frac{(D_k^hu(x)-D_k^hu(y))(\varphi^2(x)D_k^hu(x)-\varphi^2(y)D_k^hu(y))}{|x-y|^{n+2s}}\, dxdy\\
			&\qquad \qquad\leqslant \int_{{\mathbb{R}^n}} (-u+u^{p+1}) v\, dx.
		\end{split}
	\end{equation}
	Using Young's inequality, one has that 
	\begin{equation}\label{fvfdsda}
		\begin{split}
		\left|\int_{{\mathbb{R}^n}} (-u+u^{p+1}) v\, dx\right|&\leqslant \epsilon \int_{{\mathbb{R}^n}} v^2\, dx+\frac{1}{\epsilon} \int_{{\mathbb{R}^n}}(-u+u^{p+1})^2\, dx\\
		&\leqslant \epsilon C_n\left(\int_{{\mathbb{R}^n}}|\nabla u|^2\, dx+\int_{{\mathbb{R}^n}} \varphi^2\, |D_k^h(\nabla u)|\, dx\right) +\frac{\Lambda}{\epsilon} \int_{{\mathbb{R}^n}}u^2\, dx\\
		&\leqslant \frac{1}{4} \left(\int_{{\mathbb{R}^n}}|\nabla u|^2\, dx+\int_{{\mathbb{R}^n}} \varphi^2\, |D_k^h(\nabla u)|\, dx\right) +4{\Lambda} \int_{{\mathbb{R}^n}}u^2\, dx
		\end{split}
	\end{equation}
	where $\Lambda=(1+\|u\|_{L^\infty(\mathbb{R}^n)})$. From this and~\eqref{vsdvsd}, one deduces that 
		\begin{equation}
		\begin{split}\label{vdfvdsdsd}
			\frac{1}{4}\int_{{\mathbb{R}^n}}&\varphi^2 |D_k^h(\nabla u)|^2\, dx+\frac{c_{n,s}}{2} \int_{{\mathbb{R}^{2n}}}\frac{(D_k^hu(x)-D_k^hu(y))(\varphi^2(x)D_k^hu(x)-\varphi^2(y)D_k^hu(y))}{|x-y|^{n+2s}}\, dxdy\\
			&\leqslant  C_n\int_{{\mathbb{R}^n}}|\nabla u|^2\, dx +4{\Lambda} \int_{{\mathbb{R}^n}}u^2\, dx,
		\end{split}
	\end{equation}
up to relabelling $C_n>0$.

We now provide an estimate of the nonlocal term on the left-hand side of~\eqref{vdfvdsdsd}:
\begin{equation}\label{vdvsd}
	J:=  \int_{{\mathbb{R}^{2n}}}\frac{(D_k^hu(x)-D_k^hu(y))(\varphi^2(x)D_k^hu(x)-\varphi^2(y)D_k^hu(y))}{|x-y|^{n+2s}}\, dxdy.
\end{equation}
We first notice that, with obvious algebraic manipulation, we can write

\begin{equation}\label{vdsvsd}
\begin{split}
	J&=  \int_{{\mathbb{R}^{2n}}}\varphi^2(x)\frac{(D_k^hu(x)-D_k^hu(y))^2}{|x-y|^{n+2s}}\, dxdy +\int_{{\mathbb{R}^{2n}}}D_k^hu(y)\frac{(D_k^hu(x)-D_k^hu(y))(\varphi^2(x)-\varphi^2(y))}{|x-y|^{n+2s}}\, dxdy\\	
	&=:J_1+J_2.
\end{split}
\end{equation}
For $J_2$, since $\varphi\in C^\infty_0(\mathbb{R}^n, \mathbb{R})$,  using Young's inequality, we see that
 \begin{equation}
 	\begin{split}\label{vskvns}
 	|	J_2|&=  \int_{{\mathbb{R}^{2n}}}D_k^hu(y)(\varphi(x)+\varphi(y))\frac{(D_k^hu(x)-D_k^hu(y))(\varphi(x)-\varphi(y))}{|x-y|^{n+2s}}\, dxdy\\	
 		&\leqslant \frac{1}{8} \int_{{\mathbb{R}^{2n}}}(\varphi(x)+\varphi(y))^2\frac{(D_k^hu(x)-D_k^hu(y))^2}{|x-y|^{n+2s}}\, dxdy +8\int_{{\mathbb{R}^{2n}}}|D_k^hu(y)|^2\frac{(\varphi(x)-\varphi(y))^2}{|x-y|^{n+2s}}\, dxdy\\	
 		&\leqslant \frac{J_1}{2}+ C_{n,s} \int_{{\mathbb{R}^{n}}}|\nabla u|^2\, dx.
  	\end{split}
 \end{equation}
Thus, from~\eqref{vdsvsd} and~\eqref{vskvns}, it follows that 
\begin{equation}\label{bdbfd}
	J\geqslant \frac{J_1}{2}-C_{n,s} \int_{{\mathbb{R}^{n}}}|\nabla u|^2\, dx.
\end{equation}
Recalling~\eqref{vdfvdsdsd}, by combining~\eqref{vdvsd}  with~\eqref{bdbfd}, one has that 
\begin{equation}\label{sdvsd}
		\frac{1}{4}\int_{{\mathbb{R}^n}}\varphi^2 |D_k^h(\nabla u)|^2\, dx+\frac{J_1}{2}\leqslant  C_{n,s}\int_{{\mathbb{R}^n}}|\nabla u|^2\, dx +4{\Lambda} \int_{{\mathbb{R}^n}}u^2\, dx,
\end{equation}
up to relabeling $C_{n,s}>0$. From this, since $J_1\geqslant 0$, we obtain that 
\begin{equation}\label{vsdsf}
	\int_{B_{R}} |D_k^h(\nabla u)|^2\, dx\leqslant  C(n,s,p,  \|u\|_{H^1(\mathbb{R}^n)},\|u\|_{L^\infty(\mathbb{R}^n)}),
\end{equation}
and thus, owing to \cite[Chapter~5.8.2, Theorem 3]{MR2597943}, we conclude that  
\begin{equation}\label{ fdvsdvs}
	\int_{B_{R/2}} |D^2 u|^2\, dx\leqslant  C(n,s,p,  \|u\|_{H^1(\mathbb{R}^n)},\|u\|_{L^\infty(\mathbb{R}^n)}).
\end{equation}
Taking $R\rightarrow \infty$, we obtain the desired result~\eqref{H^2 estimate}.
\end{proof}

\begin{proof}[Proof of Theorem~\ref{th: existence and properity}]
In light of the qualitative  results of  solutions for equation~\eqref{main equation}
in \cite{DSVZ24}, it is immediate to deduce the assertions of part (ii) for  space dimension $n\geqslant 2$, and it suffices to show that the solutions $u_s\in H^1(\mathbb{R})$ of equation~\eqref{main equation} belong to $C^2(\mathbb{R})$ when $n=1$.

Indeed,  from Lemmata~\ref{lemma uniform bound} and~\ref{lemma H^2}, we know that $u_s\in L^\infty(\mathbb{R})\cap H^2(\mathbb{R})$.  By Sobolev embedding Theorems, one has  that $u_s\in C^{1,1/2}(\mathbb{R})$ and $u_s(|x|)\rightarrow 0$ as $|x|\rightarrow \infty $.  Moreover, we notice that $\partial_ru_s\in H^1(\mathbb{R})\cap L^\infty(\mathbb{R})$ and satisfies
$$-\Delta \partial_ru_s+(-\Delta)^s\partial_ru_s+\partial_ru_s=(p+1)u_s^p\partial_ru_s.$$
Following the proof of Lemma~\ref{lemma H^2}, we can infer that $\partial_r u_s\in H^2(\mathbb{R}) $, and exploiting the Sobolev embedding Theorems again, we derive  that $u_s\in C^{2,1/2}(\mathbb{R}).$
\end{proof}

	\smallskip
	
 We next establish some property of ground states (in the sense of Definition~\ref{ground states}):   a ground state $u_s$ has indeed Morse index equal to $1$. Note that \cite{MR3530361} shows the nondegeneracy and uniqueness results for the radial solutions with Morse index equal to $1$  in the context of single Fractional Laplacian, which is more general. The spectral assumption there that solutions have Morse index equal to $1$ seems essential for us to obtain the nondegeneracy results. As motivated by this, we raise the Question~\ref{question} in the mixed local/nonlocal setting.
	
	\begin{Proposition}\label{proposition A.1}
	Let $n\geqslant 1$ and $s\in(0,1)$.	Assume that $u_s\in\mathcal{M}_s$. Then $L_s$ has exactly one strictly negative eigenvalue. 
	\end{Proposition}
	\begin{proof}
First, we claim that  
	\begin{equation}\label{ndsvnjsd}
		\left\langle L_s\varphi, \varphi\right\rangle_{L^2}\geqslant 0\qquad \forall\varphi\bot_s u_s.
	\end{equation}
Indeed, let $\epsilon >0$, we notice that  $ u_s\bot_s \varphi $, therefore,
		\begin{equation}\label{scalar product}
			\|u_s+\epsilon\varphi\|_{s}^2 =\epsilon^2\|\varphi\|_{s}^2+\|u_s\|_{s}^2.
		\end{equation}
		Moreover, employing a Taylor expansion,  we deduce that 
		\begin{equation}\label{p+2 norm}
			\begin{split}
				&\int_{\mathbb{R}^n}|\epsilon \varphi+u_s|^{p+2}\, dx\\
				=&\int_{\mathbb{R}^n}|u_s|^{p+2}\,dx +\epsilon(p+2)\int_{\mathbb{R}^n}|u_s|^{p+1}\varphi\, dx+\frac{\epsilon^2(p+2)(p+1)}{2}\int_{\mathbb{R}^n}|u_s|^{p}\varphi^2\, dx+O(\epsilon^3).
			\end{split}
		\end{equation}
		Furthermore, by testing \eqref{main equation} against $\varphi$ and  using again that $\varphi\bot_s u_s$,  it follows that 
		\[ \int_{\mathbb{R}^n}|u_s|^{p+1}\varphi\, dx=0. \]
		As a consequence of this,  we rewrite \eqref{p+2 norm} as
		\begin{equation}
			\int_{\mathbb{R}^n}|\epsilon \varphi+u_s|^{p+2}\, dx= \int_{\mathbb{R}^n}|u_s|^{p+2}\,dx+\frac{\epsilon^2(p+2)(p+1)}{2}\int_{\mathbb{R}^n}|u_s|^{p}\varphi^2\, dx+O(\epsilon^3).
		\end{equation}
		Utilizing again the Taylor expansion, we conclude that
		\begin{equation}\label{tarloy}
		\left(\int_{\mathbb{R}^n}|\epsilon \varphi+u_s|^{p+2}\, dx\right)^{\frac{2}{p+2}}= \|u_s\|^{2}_{L^{p+2}(\mathbb{R}^n)}\left(1+\epsilon^2(p+1)\frac{\int_{{\mathbb{R}^n}}|u_s|^{p}\varphi^2\, dx}{\|u_s\|^{p+2}_{L^{p+2}(\mathbb{R}^n)}}+O(\epsilon^3)\right).
		\end{equation}
		Moreover,   the minimality of the $u_s$ for $J_s(u)$ implies that 
		\begin{equation}\label{minimality}
			0\leqslant J_s(u_s+\epsilon\varphi)-J_s(u_s)= \frac{\|u_s+\epsilon\varphi\|_{s}^2}{\|u_s+\epsilon\varphi\|^2_{L^{p+2}(\mathbb{R}^n)}}-\frac{\|u_s\|_{s}^2}{\|u_s\|^2_{L^{p+2}(\mathbb{R}^n)}}.
		\end{equation}
		By inserting \eqref{scalar product} and \eqref{tarloy} into \eqref{minimality}, one has,
		\begin{equation}
			\begin{split}
			0&\leqslant	\left(\epsilon^2\|\varphi\|_{s}^2+\|u_s\|_{s}^2\right)-\|u_s\|_{s}^2\left(1+{\epsilon^2(p+1)}\frac{\int_{{\mathbb{R}^n}}|u_s|^{p}\varphi^2\, dx}{\|u_s\|^{p+2}_{L^{p+2}(\mathbb{R}^n)}}+O(\epsilon^3)\right)\\
			&=\epsilon^2\left(\|\varphi\|_{s}^2-(p+1)\int_{\mathbb{R}^n}|u_s|^{p}\varphi^2\, dx+O(\epsilon)\right).
			\end{split}
		\end{equation}
	As a result, we can infer that 
	\begin{equation}
		\left\langle L_s\varphi, \varphi\right\rangle_{L^2}=\|\varphi\|_{s}^2-(p+1)\int_{\mathbb{R}^n}|u_s|^{p}\varphi^2\, dx\geqslant 0\qquad \forall \varphi\bot_s u_s.
	\end{equation}
This ends the claim.

Moreover, we observe that $\left\langle L_su_s, u_s\right\rangle_{L^2}<0 $ with $u_s$ is radial, 
thus, by the min-max principle, the
operator $L_s$ has exactly one negative eigenvalue and the corresponding eigenfunction is radial.
\end{proof}

	\section{The Kato class $K_s$ and Perron–Frobenius theory}

Let us first introduce a suitable class of potentials $ V $ for the mixed local/nonlocal Schr\"{o}dinger
operators discussed here. In many respects (e.g., perturbation theory and properties of eigenfunctions), the following ``Kato class” (denoted by $K_s$) is a natural choice .

\begin{Definition}\label{definition of Kato}
	Let $0<s<1$ and $n\geqslant 1$. We say that the potential $V \in K_s$ if and only if $V:\mathbb{R}^n\rightarrow \mathbb{R}^n$
	is measurable and satisfies
	\begin{equation}
	\lim\limits_{\mathbb{R}_+\ni\beta \rightarrow +\infty}\sup\limits_{x\in\mathbb{R}^n}\left|\left(-\Delta+(-\Delta)^s+\beta\right)^{-1}|V|\right|(x)=0.
	\end{equation}
\end{Definition}

In this section, we collect some basic results about mixed local/nonlocal Schr\"{o}dinger operators
$H=-\Delta+(-\Delta)^s+V$ acting on  $L^2(\mathbb{R}^n)$ by following the arguments to fractional Schr\"{o}dinger operator in \cite{MR3070568, MR3530361}.

\begin{remark}
 If $V\in K_s$, then $H=-\Delta+(-\Delta)^s+V$ defines a unique self-adjoint operator
	on $ L^2(\mathbb{R}^n)$ and the corresponding heat kernel $e^{-tH}$ maps $ L^2(\mathbb{R}^n)$ 
	into $L^2(\mathbb{R}^n)\cap C(\mathbb{R}^n)$ for any $t>0$. In particular, any $L^2$-eigenfunction of $H$ is continuous and bounded. See also \cite{MR1054115,MR670130} for equivalent definitions of $K_s$ and further background	materials.
\end{remark}

First, we derive the following sufficient condition in terms of $L^p$-spaces for a potential
$V $ to be in $K_s$.

\begin{Lemma}\label{lemma kato lp}
Let  $n\geqslant 1$, $0<s<1$ and $V:\mathbb{R}^n\rightarrow \mathbb{R}^n$ be given. Then the following holds:

If $ V\in L^q(\mathbb{R}^n) $ for some $q>\max\left\{n/2s,1\right\}$, then $V\in K_s$. 
\end{Lemma}
\begin{proof}
	In view of Definition~\ref{definition of Kato}, it suffices to show that
	\begin{equation}\label{dsacsd}
		\lim\limits_{\beta\rightarrow +\infty}\sup\limits_{x\in\mathbb{R}^n}\left|\left(-\Delta+(-\Delta)^s+\beta\right)^{-1}|V|\right|(x)=0.
	\end{equation}
Using the fact that 
\begin{equation}\label{dcdscdd}
	\left(-\Delta+(-\Delta)^s+\beta\right)^{-1}=\int_{0}^{\infty} e^{-\beta t}e^{-t\left(-\Delta+(-\Delta)^s\right)}\, dt \qquad \text{for } \beta>0
\end{equation}
  and H\"{o}lder inequality, we derive that
\begin{equation}\label{hthtr}
	\|\left(-\Delta+(-\Delta)^s+\beta\right)^{-1}|V|\|_{L^\infty}\leqslant C \|V\|_{L^q(\mathbb{R}^n)}\int_{0}^{\infty} e^{-t\beta}\|\mathcal{H}_s(\cdot,t)\|_{L^r(\mathbb{R}^n)}\, dt
\end{equation}
with $1/p+1/r=1$, where $\mathcal{H}_s$ denotes a heat kernel associated with the mixed $ -\Delta+(-\Delta)^s $ operator given by
\begin{equation}\label{heat kernel}
	\mathcal{H}_s(x,t)=\int_{{\mathbb{R}^n}} e^{-t|\xi|^2-t|\xi|^{2s}}e^{2\pi ix\cdot \xi} \, d\xi\quad \text{for } t>0, x\in\mathbb{R}^n.
\end{equation}
 We point out that the heat kernel $\mathcal{H}_s(x,t)$ may be viewed as a transition density of the L\'{e}vy
 process $X$ which is the mixture of the
 Brownian motion and an independent symmetric $2s$-stable L\'{e}vy process, and  the pseudo-differential operator $ -\Delta+(-\Delta)^s$ is the infinitesimal generator of $X$.

 From \cite{DSVZ24}, we know that,  there exists a positive constant $C_1>0$ only depending on $n,s$ such that  
 \begin{equation}\label{h_s}
 0<	\mathcal{H}_s(x,t)\leqslant C_1 \left\{\frac{t}{|x|^{n+2s}}\vee \frac{t^s}{|x|^{n+2s}}\right\}\wedge \left\{ t^{-\frac{n}{2s}}\wedge  t^{-\frac{n}{2}}\right\}
 \end{equation}
where~$a\wedge b:=\min\left\{a,b\right\}$ and~$a\vee b:=\max\left\{a,b\right\}$.
 
 Hence, we find that 
 \begin{equation}\label{vdf}
 	\|\mathcal{H}_s(\cdot,t)\|_{L^r(\mathbb{R}^n)}\leqslant C_{n,s,r}\left\{ t^{-\frac{nr-n}{2sr}}\wedge  t^{-\frac{nr-n}{2r}}\right\}.
 \end{equation}
 Since $1/p=(r-1)/r$, the previous bound for $\|\mathcal{H}_s(\cdot,t)\|_{L^r(\mathbb{R}^n)}$ implies that 
 \begin{equation}\label{sdvd}
 	\|\left(-\Delta+(-\Delta)^s+\beta\right)^{-1}|V|\|_{L^\infty}\leqslant C_{n,s,q} \|V\|_{L^q(\mathbb{R}^n)}\int_{0}^{\infty} e^{-t\beta}\left\{t^{-\frac{n}{2sq}}\wedge  t^{-\frac{n}{2rq}}\right\}\, dt.
 \end{equation}
From this and~\eqref{dsacsd}, we deduce that if $p>\left\{\frac{n}{2s}\vee 1\right\}$,  then $V\in K_s$.
\end{proof}

  From Lemma~\ref{lemma kato lp} and Theorem~\ref{th: existence and properity}, it is immediate to check that	$u_s^p\in K_s$ with $u_s\in \mathcal{M}_s$, i.e., the potential $V =u_s^p $ belongs to the ``Kato-class" with respect to
 	$ -\Delta+(-\Delta)^s $. We thus can show that mixed  Schr\"{o}dinger operators $H=-\Delta+(-\Delta)^s+V$ enjoy
the following Perron–Frobenius property.

\begin{Lemma}\label{perron-frobenius }
	Let $n\geqslant 1$, $0<s<1$ and consider $H=-\Delta+(-\Delta)^s+V$ acting on $L^2(\mathbb{R}^n)$, where we
	assume that $V\in K_s$. Suppose that $e=\inf\sigma(H)$ is an eigenvalue. Then $e$ is simple and
	its corresponding eigenfunction $\varphi=\varphi(x)$ is positive (after replacing $\varphi$ by $-\varphi$ if necessary).
\end{Lemma}
\begin{proof}
	We first claim that  the operator $e^{-t\left(-\Delta+(-\Delta)^s\right)}$ acting on $L^2(\mathbb{R}^n) $ is positivity improving
	for $t>0$. By this, we mean that if $f\geqslant 0 $ and $f\nequiv 0 $, then $[e^{-t\left(-\Delta+(-\Delta)^s\right)} f]>0$.

	Indeed, from \cite{MR1054115,MR549115}, we know that 
	\begin{equation}\label{vdsvdsdfd}
		[e^{-t\left(-\Delta+(-\Delta)^s\right)} f]=\textbf{E}_{x}\left\{f(X_t)\right\}=\int_{{\mathbb{R}^n}}f(x+y)\mathcal{H}_s(y,t)\, dy,
	\end{equation}
where $\mathcal{H}_s(x,t)$ is a heat kernel given by~\eqref{heat kernel} (transition density of a L\'{e}vy
process).	Since $\mathcal{H}_s(x,t)>0$ for $t>0$ and  $x\in \mathbb{R}^n$  (see e.g. \cite{DSVZ24}), we deduce  that $[e^{-t\left(-\Delta+(-\Delta)^s\right)} f]>0$.

	Next, we consider $H=-\Delta+(-\Delta)^s+V$ acting on $L^2(\mathbb{R}^n)$. Since $V\in K_s$, it follows that
$	V$ is an infinitesimally bounded perturbation of $-\Delta+(-\Delta)^s $ (see e.g. \cite[Section 3]{MR1054115}). Hence we can apply standard
	Perron–Frobenius-type arguments (see, e.g., \cite{MR493421}) to deduce that the largest eigenvalue
	of $e^{-tH}$ is simple and its corresponding eigenfunction strictly positive. By functional
	calculus, this fact is equivalent to saying that the lowest eigenvalue of $H$ is simple and
	has a positive eigenfunction.
\end{proof}

\begin{Lemma}\label{lemma n=1}
Let $n=1$ and $H=-\Delta+(-\Delta)^s+V$ be as in Lemma~\ref{perron-frobenius }. Assume that
$V=V(|x|)$ is even and let $H_{odd}$ denote the restriction of $H $ to $L^2_{odd}(\mathbb{R})$. If $e=\inf\sigma(H_{odd})$
is an eigenvalue, then $e$ is simple and the corresponding odd eigenfunction $\varphi=\varphi(x)$
satisfies $\varphi(x)>0$ for $x>0$ (after replacing $\varphi $ by $-\varphi$ if necessary).
\end{Lemma}
\begin{proof}
 We first claim that $\mathcal{H}_s(x,t)$ given by~\eqref{heat kernel} is  strictly decreasing for $x>0$. 
 
 Indeed, from \cite{DSVZ24}, it follows that  $\mathcal{H}_s(x,t)$ is positive,  even in $x$ and decreasing with respect to $|x|$.
 We now exploit the Bernstein’s
 theorem about the Laplace transform to show $\mathcal{H}_s(x,t)$ is strictly decreasing for $x>0$. 
 
 We observe that, by the scaling property of $\mathcal{H}_s(x,t)$, we find that 
\begin{equation}\label{vsdv}
	\begin{split}
		\mathcal{H}_s(x,t)&=t^{-\frac{1}{2s}}\int_{\mathbb{R}} e^{2\pi i t^{-\frac{1}{2s}}x\cdot \xi} e^{-(|\xi|^{2s}+t^{1-\frac{1}{s}}|\xi|^{2})}\,d\xi\\
		&=: t^{-\frac{1}{2s}}\mathcal{H}_s(t^{-\frac{1}{2s}}x,1,t^{1-\frac{1}{s}})  \quad \text{for } t>0, x\in\mathbb{R}.
	\end{split}
\end{equation} 
Hence, it suffices to show that for every positive constant $a$, $ \mathcal{H}_s(x,1,a)$ is strictly decreasing for $x>0$.  

For this, we consider the non-negative function $g(z)=z^s+az$ on the half-line $[0,\infty)$ . Since
 $0<s<1$, it is easy to check that $g(z)$ is completely monotone (i.e. $(-1)^m\frac{d^m}{dz^m} g(z)\leqslant 0 $ for all $m\in\mathbb{N}$).  This implies that the map  $z\rightarrow e^{-g(z)}$ is completely monotone
 as well. Hence, by Bernstein’s theorem, we infer that
 \[ e^{-g(z)}=\int_{0}^{+\infty} e^{-\tau z}d\mu_g(\tau),\]
  for some non-negative finite Borel measure $ \mu_g $ depending on $g$.
  Setting $z=|\xi|^2$, and recalling the inverse Fourier transform of the Gaussian $e^{-\tau|\xi|^2}$,  we  obtain that
  \begin{equation}\label{ksdvsd}
  	\mathcal{H}_s(x,1,a)= \int_{0}^{+\infty} \frac{1}{(\pi \tau)^{1/2}} e^{-\frac{x^2}{4\tau}}d\mu_s(\tau)
  \end{equation}
  with some non-negative measure $\mu_s\geqslant 0$ and $\mu_s\nequiv 0$. From~\eqref{ksdvsd}, one has that  $\mathcal{H}_s(x,1,a)>0$ for  every $a>0$ and $x\in \mathbb{R}$ and
  \begin{equation}\label{sdvsvsd}
  	\frac{d}{dx}\mathcal{H}_s(x,1,a)<0 \qquad \text{for every } x>0 \text{ and } a>0.
  \end{equation}
 	This implies the claim. 
  
 Next,  let $\left(-\Delta+(-\Delta)^s\right)_{odd}$ denote the restriction of $-\Delta+(-\Delta)^s$ to $L^2_{odd}(\mathbb{R})$. By odd symmetry, we find
  that $e^{-t\left(-\Delta+(-\Delta)^s\right)_{odd}}$ acts on $f\in L^2_{odd}(\mathbb{R})$ according to
  \begin{equation}\label{vsvvsd}
  	[e^{-t\left(-\Delta+(-\Delta)^s\right)_{odd}} f](x)=\int_{0}^{+\infty}\left(\mathcal{H}_s(x-y,t)-\mathcal{H}_s(x+y,t)\right)f(y)\, dy.
  \end{equation}

  Moreover, in view of the claim that $\mathcal{H}_s(x,t)$ is  strictly decreasing for $x>0$, it is easy to check that  
 \begin{equation}\label{vfsbfs}
 	\mathcal{H}_s(x-y,t)-\mathcal{H}_s(x+y,t)>0 \qquad \text{for } 0<x,y<+\infty.
 \end{equation}

Therefore, by combining~\eqref{vsvvsd} with~\eqref{vfsbfs}, we deduce that $$[e^{-t\left(-\Delta+(-\Delta)^s\right)_{odd}} f](x)>0.$$  This implies that $e^{-t\left(-\Delta+(-\Delta)^s\right)_{odd}}$ can be identified
with a positivity improving operator on $ L^2_{odd}(\mathbb{R})$.

Now, we consider $H_{odd}=\left(-\Delta+(-\Delta)^s\right)_{odd}+V$ with $V =V (|x|)$ even. Using standard Perron–
Frobenius arguments (see the proof of Lemma~\ref{perron-frobenius } and the reference there), we deduce
that the largest eigenvalue of $e^{-tH_{odd}}$ on $L^2_{odd}(\mathbb{R})$ is simple and its corresponding eigenfunction
satisfies $\varphi=\varphi(x)>0$ for $x>0$. By functional calculus, this fact now implies Lemma~\ref{lemma n=1}.
\end{proof} 

\section{Asymptotics of Eigenfunction}\label{sec:Asymptotics of Eigenfunctions}
The following result provides some uniform estimates regarding the spatial decay of eigenfunction of $(-\Delta)+(-\Delta)^s+V$ below the essential spectrum. Indeed, the following estimates can be found in \cite{MR3530361,DSVZ24a} without, however any direct insight into uniformity of these estimates with respect to $s$ close to $0$.

\begin{Lemma}\label{lemma uniform decay of spectrum}
	Let $n\geqslant 1$ and $s\in(0,1)$,  and suppose that $V\in L^\infty(\mathbb{R}^n)$ with $V(x)\rightarrow 0$ as $|x|\rightarrow\infty $. Assume that $u\in H^1(\mathbb{R}^n)$ with $\|u\|_{H^1(\mathbb{R}^n)}=1$ satisfies 
	\begin{equation}\label{ebet}
		-\Delta u+(-\Delta)^su+Vu=Wu
	\end{equation} with some $W<0$. Furthermore, let $0<\lambda<-W$  be given and suppose that $R>0$ is such that $V(x)+\lambda\geqslant 0$ for every $|x|\geqslant R$, then the following properties hold:
	\begin{itemize}
		\item[(i)] There exists a positive constant $C_1>0$ independing of $s$, such that for every ${s\in(0,1)}$
		\begin{equation}\label{svds}
			\|u\|_{L^\infty(\mathbb{R}^n)}\leqslant C_1.
		\end{equation}
	\item[(ii)] For all $|x|\geqslant 1$, it holds that 
	\begin{equation}\label{sdvbsd}
		|u(x)|\leqslant C\left(s,W,\lambda,n,R,\|V\|_{L^\infty(\mathbb{R}^n)},\|u\|_{L^\infty(\mathbb{R}^n)}\right)|x|^{-(n+2s)}.
	\end{equation}
\item[(iii)] There exists a positive constant $s_1$ such that, for all $|x|\geqslant 1$ and $s\in(0,s_1]$ it holds that 
\begin{equation}\label{sdvbsdnk}
	|u(x)|\leqslant C\left( \lambda,n,R,W,\|V\|_{L^\infty(\mathbb{R}^n)},\|u\|_{L^\infty(\mathbb{R}^n)}\right)|x|^{-n}.
\end{equation}
	\end{itemize}
\end{Lemma}
\begin{proof}
	We start by proving part (i). Following closely the proof of Lemma~\ref{lemma uniform bound},  for some $k>0$, let us denote
	\begin{equation}
		Q_{k}=\begin{cases}
			u_+ \quad &\text{ if } u<k\\
			k &\text{ if } u\geqslant k.
		\end{cases}
	\end{equation}
	We then observe that $D Q_{k}=0$ in $\left\{u\geqslant k\right\}\cup \left\{u\leqslant 0\right\}$.
	Set the text function 
	\begin{equation}\label{sgsd}
		\varphi:= Q_{k}^\alpha u=Q_{k}^\alpha u_+\in H^1(\mathbb{R}^n)
	\end{equation}
	for suitable constant  $\alpha\geqslant 0$. 
	Direct calculation yields that for every $ x, y\in\mathbb{R}^n$
	\begin{equation}\label{sdad}
		\varphi(x)\leqslant \varphi(y) \quad \text{if } u(x)\leqslant u(y), 
	\end{equation}
	\begin{equation}\label{varphid}
		\text{ and }\quad 	D\varphi(x)= Q_{k}^\alpha(x) D (u_+)(x)+\alpha Q_{k}^{\alpha}(x) D Q_{k}(x).
	\end{equation}
	
	We multiply the equation of~\eqref{ebet} by $\varphi$ and integrate over $\mathbb{R}^n$, one derives that
	\begin{equation}
			\int_{\mathbb{R}^n} Q_{k}^\alpha(x)|D {u_+}(x)|^2 \, dx+\int_{\mathbb{R}^n}  \alpha Q_{k}^{\alpha}(x) |D Q_{k}(x)|^2\, dx
			\leqslant(\|V\|_{L^\infty(\mathbb{R}^n)}+|W|)\int_{\mathbb{R}^n}  Q_{k}^\alpha(x) u_+^2(x)\, dx.
	\end{equation}
	Moreover, we denote $w:=Q_{k}^{\frac{\alpha}{2}}u_+$. Thus, one has that 
	\begin{equation}
			\int_{\mathbb{R}^n} |Dw|^2\, dx\leqslant \left(1+\alpha\right)(\|V\|_{L^\infty(\mathbb{R}^n)}+|W|)\int_{\mathbb{R}^n}  Q_{k}^\alpha(x) u_+^2(x)\, dx\leqslant C_{W,V}\left(1+\alpha\right)\int_{\mathbb{R}^n}  w^2(x)\, dx.
	\end{equation}

Using the Gagliardo-Nirenberg interpolation inequality, it follows that 
\begin{equation}\label{chjv2}
	\left(\int_{\mathbb{R}^n} |w|^{q_0}\, dx\right)^{\frac{2}{q_0}}\leqslant C(1+\alpha)^{\frac{n}{2}-\frac{n}{q_0}}\int_{\mathbb{R}^n}|w|^2\, dx
\end{equation} 
	where $q_0$ is either equal to the critical exponent $2^*$ if $n > 2$ or any real number in $(2,+\infty)$
if $n =1,2$, for renaming $C$ 	which only depends on  $n,q_0,W,V$.
	
	 Furthermore, denote $\beta:=\frac{n}{2}-\frac{n}{q_0}>0 $ and $\gamma:=\frac{q_0}{2}>1$, thus the inequality~\eqref{chjv2} can be rewritten as
	\begin{equation}\label{dsak}
		\left(\int_{\mathbb{R}^n} |w|^{2\gamma}\, dx\right)^{1/\gamma}\leqslant C(1+\alpha)^{\beta}\int_{\mathbb{R}^n}|w|^2\, dx.
	\end{equation} 
Hence,	the definition of $w$ implies that 
	\begin{equation}\label{sdfsng}
		\left(\int_{\mathbb{R}^n} |Q_{k}|^{(\alpha+2)\gamma}\, dx\right)^{1/\gamma}\leqslant C(1+\alpha)^{\beta}\int_{\mathbb{R}^n}|u_+|^{\alpha+2}\, dx.
	\end{equation} 
	Moreover, set $\eta:=\alpha+2\geqslant 2$, we have that 
	\begin{equation}\label{dvdsfn}
		\left(\int_{\mathbb{R}^n} |Q_{k}|^{\eta\gamma}\, dx\right)^{1/\gamma}\leqslant C(\eta-1)^{\beta}\int_{\mathbb{R}^n}|u_+|^{\eta}\, dx,
	\end{equation} 
	provided the integral in the right hand side is bounded. Here the constant $C$ is independent of $\eta$ and $k$.  We are on a position of taking $k\rightarrow \infty$, obtaining that 
	\begin{equation}\label{sdadasn}
		\|u_+\|_{L^{\eta\gamma}(\mathbb{R}^n)}\leqslant C^{1/\eta}\eta^{\beta/\eta}\|u_+\|_{L^\eta(\mathbb{R}^n)},
	\end{equation} 
	provided $\|u_s\|_{L^{\eta}(\mathbb{R}^n)}$ is bounded. The above estimate allows us to iterate, beginning with $\eta=2$. Now set for $i=0,1,2,3,\cdots$, $\eta_0=2$ and $\eta_{i+1}= \eta_i\gamma$.
	
	As a result, by iteration we have that
	\begin{equation}\label{vjhm}
		\|u_+\|_{L^{\eta_{i+1}}(\mathbb{R}^n)}\leqslant C^{\sum\frac{i}{\gamma^i}}\|u_+\|_{L^2(\mathbb{R}^n)},
	\end{equation} 
	where $C$  depends on $n,W,V,\gamma,\beta$. Since $\gamma>1$, one has that $\sum_{i=0}^{\infty}\frac{i}{\gamma^i}$ is bound. Let $i\rightarrow \infty$, we conclude that 
	\[ \|u_+\|_{L^\infty(\mathbb{R}^n)}\leqslant C\|u_+\|_{L^2(\mathbb{R}^n)}\leqslant C\|u\|_{L^2(\mathbb{R}^n)}, \]
	for renaming $C$ only depends on $n,W,V,\gamma,\beta$.
We observe that,	using a similar argument,  we  obtain that $ \|u_-\|_{L^\infty(\mathbb{R}^n)}\leqslant C\|u\|_{L^2(\mathbb{R}^n)}.$ This yields the desired result~\eqref{svds}

\smallskip

We next prove part (ii). For this, let $s\in(0,1)$ be fixed, in virtue of Lemmata~\ref{lemma uniform bound} and~\ref{lemma H^2}, we obtain that 
$\|u\|_{L^\infty(\mathbb{R}^n)}\leqslant C\left(n,W,\|V\|_{L^\infty(\mathbb{R}^n)}\right)$ and $u\in H^2(\mathbb{R}^n)$.
Furthermore, for any $f\in H^{2}(\mathbb{R}^n)$, owing to \cite[Lemma C.2]{MR3530361},  we have the general (Kato-type) inequaltiy 
\begin{equation}\label{vhj}
	(-\Delta)|f|+(-\Delta)^s|f|\leqslant ({\rm sgn }\, f)\left(-\Delta f+(-\Delta)^s f\right)\quad \text{a.e. on } \mathbb{R}^n 
\end{equation} 
{Here, the function $ ({\rm sgn }\, f)=1$ if $f>0$, $ ({\rm sgn }\, f)=-1$ if $f<0$, and $ ({\rm sgn }\, f)=0$ if $f=0$.}

 Hence, by assumption, we see that $0<\lambda<-W $ and $R>0$ such that $V(x)+\lambda\geqslant 0$ for  $|x|\geqslant R$, 
using~\eqref{vhj} on the set $\mathbb{R}^n\backslash B_R$, we derive that 
\begin{equation}\label{gege}
	(-\Delta)|u|+(-\Delta)^s|u|+\lambda|u|\leqslant 0 \qquad \text{a.e. on } \mathbb{R}^n\backslash B_R.
\end{equation}
Now we claim that 
\begin{equation}\label{bhjkj}
	|u|(x)\leqslant  \frac{C(s,\lambda,R,n,\|u\|_{L^\infty(\mathbb{R}^n)})}{|x|^{n+2s}} \quad  \text{on } \mathbb{R}^n\backslash B_R.
\end{equation}
Indeed, it follows from a comparison arguments.  
We denote the fundamental function $\mathcal{K}_{\lambda}$  for the $-\Delta +(-\Delta)^s+\lambda$ in the whole space by the Fourier transform 
\begin{equation}\label{dvs}
	\hat{\mathcal{K}}_{\lambda}(y)=\frac{1}{|y|^2+|y|^{2s}+\lambda}.
\end{equation}
Thus, we have that 
$\left(-\Delta +(-\Delta)^s+{1}/{2} \right)\mathcal{K}_{\lambda}(x)=\delta_0$ in $\mathbb{R}^n$. From \cite[Lemma~4.4]{DSVZ24}, it follows that  $\mathcal{K}_{\lambda}(x)\geqslant c>0$ for $ |x|\leqslant R$ with some suitable constant $c=c(n,s,R,\lambda)$. Let us take $C_0:=\|u\|_{L^\infty(\mathbb{R}^n)}c^{-1}>0$, this implies that $C_0 \mathcal{K}_{\lambda}(x)\geqslant |u|(x)$ in $B_R$. From this, we define the function 
\begin{equation}\label{reger}
	\omega(x):=C_0 \mathcal{K}_{\lambda}(x) -|u|(x).
\end{equation}
We notice that 
\begin{equation}\label{vsdd}
	(-\Delta)\omega+(-\Delta)^s\omega+\frac{1}{2}\omega\geqslant 0 \qquad \text{a.e. on } \mathbb{R}^n\backslash B_R.
\end{equation}
From Lemma~3.1 in \cite{DSVZ24a}, we have that $\omega\geqslant 0$ in $ \mathbb{R}^n$. 

As a result, using the fact that $\omega$ is continuous away from the origin, we obtain that $\omega\geqslant 0$ on $\mathbb{R}^n$ which implies that  
\begin{equation}\label{vsdvsdds}
	|u|(x)\leqslant C_0\mathcal{K}_{\lambda}(x) \qquad \text{in } \mathbb{R}^n. 
\end{equation}
Therefore, the desired result~\eqref{bhjkj} follows from \cite[Lemma~4.4]{DSVZ24}, which completes this proof of part (i).

 \smallskip
 
To show part (iii), since $	\mathcal{K}_{\lambda}(x)$ is decreasing in $r=|x|$, it suffices to show  that for every $s\in(0,1)$
 \begin{equation}\label{bdfbd}
 	\mathcal{K}_{\lambda}(x)=\int_{0}^{\infty} e^{-\lambda t}\, \mathcal{H}_s(x,t)\, dt\leqslant C_{n,\lambda}|x|^{-n}\qquad \text{for every } |x|>1, 
 \end{equation}
and for $s>0$ sufficiently small,
\begin{equation}\label{bdfbdfdv}
	\mathcal{K}_{\lambda}(x)=\int_{0}^{\infty} e^{-\lambda t}\, \mathcal{H}_s(x,t)\, dt\geqslant C_{n,R,\lambda}\qquad \text{for every } |x|=R,
\end{equation}
where $\mathcal{H}_s(x,t)$ is given in~\eqref{heat kernel}.

Step 1. We prove that the upper bound on $	\mathcal{K}_{\lambda}(x)$ in~\eqref{bdfbd}. Indeed, from Lemma~A.3 in \cite{DSVZ24}, one has that 
\begin{equation}\label{fdbadf}
	\begin{split}
	|x|^{n}\mathcal{H}_s(x,t)&={(2\pi)^{\frac{n}{2}}}\int_{0}^{+\infty}t e^{-t\frac{r^{2s}}{|x|^{2s}}-t\frac{r^{2}}{|x|^{2}}} \left(\frac{2sr^{\frac{n}{2}+2s-1}}{|x|^{2s}}+\frac{2r^{\frac{n}{2}+1}}{|x|^{2}}\right) J_{\frac{n}{2}}(r)\, dr\\
	&= {(2\pi)^{\frac{n}{2}}} \text{Re}\int_{0}^{+\infty} t e^{-t\frac{r^{2s}}{|x|^{2s}}-t\frac{r^{2}}{|x|^{2}}} \left(\frac{2sr^{\frac{n}{2}+2s-1}}{|x|^{2s}}+\frac{2r^{\frac{n}{2}+1}}{|x|^{2}}\right) H^{(1)}_{\frac{n}{2}}(r)\, dr\\
	&={(2\pi)^{\frac{n}{2}}} \text{Re}\int_{L_1}t e^{-t\frac{z^{2s}}{|x|^{2s}}-t\frac{z^{2}}{|x|^{2}}} \left(\frac{2sz^{\frac{n}{2}+2s-1}}{|x|^{2s}}+\frac{2z^{\frac{n}{2}+1}}{|x|^{2}}\right) H^{(1)}_{\frac{n}{2}}(z)\, dz
	\end{split}
\end{equation}
where~$ H^{(1)}_{\frac{n}{2}}(z)$ is the Bessel function of the third kind, Re~$A$ denotes the real part of~$A$ and $L_1:=\left\{z\in\mathbb{C} : \arg z=\frac{\pi}{6}\right\}$.

Hence, one derives that 
\begin{equation}\label{fdbadfvj}
	\begin{split}
		|x|^{n}\mathcal{H}_s(x,t)=	& {(2\pi)^{\frac{n}{2}}}\text{Re}\int_{0}^{+\infty}te^{-t\frac{\left(re^{i\frac{\pi}{6}}\right)^{2s}}{|x|^{2s}}-t\frac{\left(re^{i\frac{\pi}{6}}\right)^{2}}{|x|^{2}}} \left({2s\left(re^{i\frac{\pi}{6}}\right)^{\frac{n}{2}+2s-1}}|x|^{-2s}+2{\left(re^{i\frac{\pi}{6}}\right)^{\frac{n}{2}+1}}{|x|^{-2}}\right) H^{(1)}_{\frac{n}{2}}\left(re^{i\frac{\pi}{6}}\right)e^{i\frac{\pi}{6}}\, dr\\
	\leqslant & {(2\pi)^{\frac{n}{2}}}\int_{0}^{+\infty}t \left({2sr^{\frac{n}{2}+2s-1}}|x|^{-2s}+2{r^{\frac{n}{2}+1}}{|x|^{-2}}\right) \left|H^{(1)}_{\frac{n}{2}}\left(re^{i\frac{\pi}{6}}\right)\right|\, dr.
	\end{split}
\end{equation}
Furthermore,  owing to~(A.13) in \cite{DSVZ24}, let $z\in L_1$, one has that
\begin{equation}\label{fv}
	\left|H^{(1)}_{\frac{n}{2}}\left(z\right)\right|\leqslant 2\int_{0}^{+\infty}e^{-\frac{r}{4}e^\tau}e^{\frac{n\tau}{2}}\,d\tau.
\end{equation}
From the above two formulas, one has that, for every $|x|>1$ and $t>0$
\begin{equation}\label{fdbadfvjm}
	\begin{split}
		|x|^{n}\mathcal{H}_s(x,t)
		\leqslant &2t {(2\pi)^{\frac{n}{2}}}\int_{0}^{+\infty} \left({2sr^{\frac{n}{2}+2s-1}}+2{r^{\frac{n}{2}+1}}\right) \int_{0}^{+\infty}e^{-\frac{r}{4}e^\tau}e^{\frac{n\tau}{2}}\,d\tau\, dr\\
		\leqslant &2t {(2\pi)^{\frac{n}{2}}}\left(4^{\frac{n}{2}+2s}\Gamma\left(\frac{n}{2}+2s\right)+4^{\frac{n}{2}+2}\Gamma\left(\frac{n}{2}+2\right)\right)\leqslant tC_n.
	\end{split}
\end{equation}
This implies that  for every $|x|>1$,
\begin{equation}\label{fdgfd}
\mathcal{K}_{\lambda}(x)=\int_{0}^{\infty} e^{-\lambda t}\, \mathcal{H}_s(x,t)\, dt\leqslant	C_n \int_{0}^{\infty} e^{-\lambda t}\,t|x|^{-n}\leqslant C_{n,\lambda}|x|^{-n},
\end{equation}
which conclude the proof of~\eqref{bdfbd}.

Step 2. We show that the uniform lower bound on $	\mathcal{K}_{\lambda}(x)$  in~\eqref{bdfbdfdv} for $s>0$ sufficiently small. For this, from Lemma~A.5 in \cite{DSVZ24}, one has that 
\begin{equation}
		\mathcal{H}_s(x,t) =e^{-\frac{\pi|x|^2}{t}}t^{-\frac{n}{2}}\int_{\mathbb{R}^n}e^{-|y-i\pi xt^{-\frac{1}{2}}|^2}e^{-t^{1-s}|y|^{2s}}\,dy.
\end{equation}
Furthermore, we denote 
\begin{equation}\label{fbd}
	\Lambda:=e^{-t}\int_{\mathbb{R}^n}e^{-|y-i\pi xt^{-\frac{1}{2}}|^2}
	\, dy=e^{-t}\frac{\omega_n\Gamma(\frac{n}{2})}{2}
\end{equation}
thanks to the Residue Theorem. Hence, for any $t\in(1,2)$ and $|x|=R$, we have that 
	\begin{equation*}
	\begin{split}
		\left|	\frac{\mathcal{H}_s(x,t)}{e^{-\pi\frac{|x|^2}{t}}t^{-\frac{n}{2}}} -\Lambda\right|&=\left|\int_{\mathbb{R}^n}e^{-|y-i\pi xt^{-\frac{1}{2}}|^2}\left(e^{-t}-e^{-t^{1-s}|y|^{2s}}\right)\,dy\right|\\
		&\leqslant t\int_{\mathbb{R}^n}e^{-|y|^2+{\pi\frac{|x|^2}{t}}}\left|(t^{-1/2}|y|)^{2s}-1\right|\,dy\leqslant 2e^{\pi{R^2}}\int_{\mathbb{R}^n}e^{-|y|^2}\left|(t^{-1/2}|y|)^{2s}-1\right|\,dy \\
		& \leqslant 2e^{\pi{R^2}} 2s\left(\int_{|y|\geqslant t^{1/2} }e^{-|y|^2}(t^{-1/2}|y|)^{2s}\ln(t^{-1/2}|y|)\,  dy+ \int_{|y|< t^{1/2} }e^{-|y|^2}\ln(t^{1/2}|y|^{-1})\,  dy\right)\\ 
		&\leqslant 2e^{\pi{R^2}} 2s\left(t^{-s-1/2}\int_{|y|\geqslant t^{1/2} }e^{-|y|^2}|y|^{2s+1}\,  dy+ \int_{|y|< t^{1/2} }e^{-|y|^2} 4e^{-1}(t^{1/2}|y|^{-1})^{1/4}\,  dy\right)\\ 
			&\leqslant 2e^{\pi{R^2}} 2s\omega_{n-1}\left(\Gamma\left(\frac{n}{2}+2s\right)+8e^{-1}2^{\frac{n}{2}-\frac{1}{8}}(n-1/4)^{-1}\right)\leqslant s C_{n,R}.
	\end{split}
\end{equation*}
From this and~\eqref{fbd}, taking $s_1=(4C_{n,R})^{-1}e^{-2}{\omega_n\Gamma(\frac{n}{2})}$,  one has that for every $s\leqslant s_1$, $t\in(1,2)$ and $|x|=R$
\begin{equation*}
		\frac{\mathcal{H}_s(x,t)}{e^{-\pi\frac{|x|^2}{t}}t^{-\frac{n}{2}}} \geqslant \Lambda-s C_{n,R}\geqslant e^{-2}\frac{\omega_n\Gamma(\frac{n}{2})}{4}.
\end{equation*}
Recall that $\mathcal{H}_s(x,t) $ is nonnegative, we obtain that for every $|x|=R$, 
\begin{equation}\label{bdfbdfdvd}
	\begin{split}
		\mathcal{K}_{\lambda}(x)&=\int_{0}^{\infty} e^{-\lambda t}\, \mathcal{H}_s(x,t)\, dt\geqslant e^{-2}\frac{\omega_n\Gamma(\frac{n}{2})}{4}\int_{1}^{2} e^{-\lambda t}e^{-\pi\frac{|x|^2}{t}}t^{-\frac{n}{2}}\, dt\\
		&= e^{-2}\frac{\omega_n\Gamma(\frac{n}{2})}{4}\int_{R^2/2}^{R^2} e^{-\lambda R^2y^{-1}}e^{-\pi y}R^{-{n}+{2}}y^{\frac{n}{2}-2}\, dy= C_{n,\lambda,R}>0.
	\end{split}
\end{equation}
This implies the desired result~\eqref{bdfbdfdv}.
\end{proof}

\bibliographystyle{is-abbrv}

\bibliography{manuscript}

	\end{document}